\documentclass[11pt, reqno]{amsart}

\usepackage{amsmath}
\usepackage{amsthm}
\usepackage{amssymb}
\usepackage{amsfonts}
\usepackage[mathscr]{eucal}
\usepackage{latexsym}
\usepackage{graphicx}
\graphicspath{{plots/}}

\usepackage{xcolor}

\newtheorem{prop}{Proposition}[section]
\newtheorem{cor}[prop]{Corollary}
\newtheorem{thm}[prop]{Theorem}
\newtheorem{lem}[prop]{Lemma}
\newtheorem{defi}[prop]{Definition}

\theoremstyle{definition}

\newtheorem{example}[prop]{Example}
\newtheorem{examples}[prop]{Examples}
\newtheorem{rem}[prop]{Remark}
\newtheorem{rems}[prop]{Remarks}

\renewcommand{\phi}{\varphi}

\newcommand{\B}{\mathbb{B}}
\newcommand{\C}{\mathbb{C}}
\newcommand{\D}{\mathbb{D}}
\newcommand{\N}{\mathbb{N}}
\newcommand{\R}{\mathbb{R}}

\newcommand{\PP}{\mathbb{P}}

\DeclareMathOperator{\supp}{supp}
\DeclareMathOperator{\tdeg}{tdeg}

\DeclareMathOperator{\dist}{dist}
\DeclareMathOperator{\spann}{span}

\DeclareTextFontCommand{\textnf}{\normalfont}

\newcommand{\diag}{\mathrm{diag}}

\renewcommand{\setminus}{\smallsetminus}
\newcommand{\ol}{\overline}

\renewcommand{\subset}{\subseteq}

\title
  [Christoffel-Darboux kernels. I]
  {Perturbations of Christoffel-Darboux kernels. I:  detection of outliers}

\author[Beckermann]{Bernhard Beckermann}
\address{Laboratoire Painlev\'e UMR 8524\\
         Department of Mathematics\\
         Universit\'e de Lille\\
         59655 Villeneuve d'Ascq, France}
\email{Bernhard.Beckermann@univ-lille.fr}

\author[Putinar]{Mihai Putinar}
\address{Department of Mathematics\\
         University of California at Santa Barbara\\
         Santa Barbara, California\\
         93106-3080\\
         {\rm and} \ \ School of Mathematics and Statistics\\
         Newcastle University\\
         Newcastle upon Tyne\\
         NE1 7RU, \ UK}
\email{mputinar@math.ucsb.edu, \ mihai.putinar@ncl.ac.uk}

\author[Saff]{Edward B.\ Saff}
\address{Center for Constructive Approximation, Department of Mathematics\\ Vanderbilt University\\
         1326 Stevenson Center\\
         37240 Nashville, TN\\
         USA}
         \email{edward.b.saff@Vanderbilt.edu}
\urladdr{http://my.vanderbilt.edu/edsaff/}

\author[Stylianopoulos]{Nikos Stylianopoulos}
\address{Department of Mathematics and Statistics,
         University of Cyprus, P.O. Box 20537, 1678 Nicosia, Cyprus}
\email{nikos@ucy.ac.cy}
\urladdr{http://ucy.ac.cy/\textasciitilde nikos}

\thanks{
{\it Acknowledgements.}
The first author was supported in part  by the Labex CEMPI (ANR-
11-LABX-0007-01).
The third author was partially supported by the
U.S.\ National Science Foundation grant  DMS-1516400.
The forth author was supported by the University of Cyprus grant 3/311-21027.
}

\date{\today}
\begin{document}


\begin{abstract} Two central objects in constructive approximation, the Chri\-stoffel-Darboux kernel and the Christoffel function, are encoding ample information about the
associated moment data and ultimately about the possible generating measures. We develop a multivariate theory of the Chri\-stoffel-Darboux kernel in $\C^d$,
with emphasis on the perturbation of Christoffel functions and their level sets with respect to perturbations of small norm or low rank. The statistical notion of leverage score provides a quantitative
criterion for the detection of outliers in large data. Using the refined theory of Bergman orthogonal polynomials, we illustrate the main results, including some numerical simulations, in the case
of finite atomic perturbations of area measure of a 2D region. Methods of function theory of a complex variable and (pluri)potential theory are widely used in the derivation of our perturbation formulas.

\end{abstract}

\maketitle

\tableofcontents

\section{Introduction}  The wide scope of this first article in a series is a systematic study of perturbations of multivariate Christoffel-Darboux kernels. We depart from a Bourbaki style by treating in parallel
a specific application, namely the detection of outliers in statistical data. Both themes have a rich and glorious history which cannot be condensed in a single research article. We will merely provide glimpses with precise references into the past or recent achievements as our story unfolds, continuously interlacing the two themes.

The reproducing kernel $K^\mu_n(z,w)$ on the space of univariate or multivariate polynomials of degree less than or equal to $n$, in the presence of an inner product derived from the Lebesgue space $L^2(\mu)$
associated with a positive measure $\mu$ in $\C^d$, bears the name of {\it Christoffel-Darboux kernel}. It is given in the univariate
case by
$$ K^\mu_n(z,w) = \sum_{j=0}^n p^\mu_j(z) \ol{p^\mu_j(w)},
$$
where the $p_j^{\mu}$ are orthonormal polynomials of respective degrees $j$.
For more than a century this object has repeatedly appeared (mostly in the one complex variable setting) as a central figure in interpolation, approximation, and asymptotic expansion
problems.

\subsection{The Christoffel-Darboux kernel and Christoffel function} Consider for instance the wide class of inverse problems leading in the end to the classical moment problem on the real line. There one has a grasp on $K_n(z,w) = K^\mu_n(z,w), \ n \geq 1,$  before knowing
the representing measure(s) $\mu$. This gives for complex and non-real values $z \in \C$ the radius of Weyl's circle of order $n$:
$$ r_n(z) = \frac{1}{|z - \overline{z}| K_n(z,z)}.$$
The vanishing of the limit $r_\infty(z) = \lim_n r_n(z)$ for a single non-real $z$ resolves without ambiguity the delicate question whether the moment problem is determinate (has a unique solution).
For a real value $x \in \R$ and a fixed $n$ the reciprocal
$\frac{1}{K_n(x,x)},$
also known as the {\it Christoffel function}, gives a tight upper bound for the maximal mass a finite atomic measure with prescribed moments up to degree $n$ can charge at the point $x$. An authoritative account
of the central role of the reproducing kernel method in moment problems and the theory of orthogonal polynomials on the line or on the circle is contained in Freud's monograph \cite{Freud}. For an enthusiastic and informative eulogy of Freud's predilection for the kernel method see \cite{Nevai}.

Much happened meanwhile. On the theoretical side, Totik \cite{Totik} has settled a long standing open problem concerning the limiting values of $\frac{n}{K_n(x,x)}$ for a positive measure
on the line, culminating a long series of particular cases originating a century ago in the pioneering work of Szeg\H{o}, see also the survey \cite{Simon1}. Roughly speaking, the limit of $\frac{1}{K_n(x,x)}$
detects the point mass charge at the point $x$ in the support of the original measure $\mu$, while $\frac{n}{K_n(x,x)}$ (seen as a function of $x$) gives the Radon-Nikodym derivative of $\mu$ with respect to
   the equilibrium measure of the support of orthogonality.
One step further, Lubinsky \cite{Lubinsky} has studied a universality principle for the asymptotics of the kernel $K^\mu_n(x,y)$, highly relevant for the theory of random matrices.
The weak-* limits of the counting measures of the eigenvalues of truncations of (unitary or self-adjoint) random matrices were recently unveiled via  similar kernel methods by Simon \cite{Simon}.

All the above results refer to a single real variable. There is, however, notable progress in the study of the asymptotics of the Christoffel-Darboux kernel in the multivariate case cf. \cite{BDM,Kroo-Lubinsky-1,Kroo-Lubinsky-2,Xu}.
These studies cover specific measures, in general invariant under a large symmetry group, the main theorems offering kernel asymptotics at points belonging to the support of the original measure. The behavior of the multivariate Christoffel-Darboux kernel outside the support of the generating measure is decided in most cases by a classical extremal problem in pluripotential theory. We shall provide details on the latter in the second section of this article and we will record a few illustrative cases in the last section.

\subsection{Detecting outliers and anomalies in statistical data}
But that's not all. Scientists outside the strict boundaries of the mathematics community continue to rediscover the wonders of Christoffel-Darboux type kernels. For a long time statisticians used the kernel method to estimate the probability density
of a sequence of identically distributed random variables \cite{Parzen}, or to separate observational data by the level set of a positive definite kernel, a technique originally developed for pattern recognition \cite{Vapnik}. The last couple of years have witnessed the entry of Christoffel-Darboux kernel through the front door of the very dynamical fields of Machine Learning and Artificial Intelligence \cite{Malyshkin1,Malyshkin2,LP1,LP2}.
Notice that in these applications $\mu$ is typically an atomic (empirical) measure with finite support in $\mathbb C^d$ representing a point cloud, which makes analysis like kernel asymptotics more involved.

In the case $d=1$, Malyshkin \cite{Malyshkin1} suggested the use of modified moments, that is, computing the Christoffel-Darboux kernel of a measure $\mu$ from orthogonal polynomials of a similar but classical measure (Jacobi, Laguerre, Hermite), without going through monomials, an idea also used in the modified Chebyshev algorithm of Sack and Donovan \cite{Gautschi}. Some numerical evidence is provided for numerical stability of this approach, even for high degrees $n$.
The approximation of the support of the original measure and detection of outliers became thus possible through the investigation of the level sets of the
Christoffel function \cite{LP1,LP2}. Not unrelated, three of the authors of this article have successfully used the Christoffel function for reconstruction, from indirect measurements in 2D, of the boundaries of an
archipelago of islands each carrying  the same constant density against area measure \cite{GPSS}.

The aim of this first essay in a series is to study the effect of various perturbations on Christoffel-Darboux kernels associated to positive measures compactly supported by $\C^d$.
An important thesis one can draw from classical constructive function theory is that even purely real questions ask for the domain of the reproducing kernel, potential functions, spectra of associated operators to be extended to complex values.
We therefore suggest to consider multivariate orthogonal polynomials in $\mathbb C^d$, where in applications real point clouds could be imbedded in $\mathbb C^d$ by using the isomorphism with $\mathbb R^{2d}$. As we will recall in \eqref{def_kernel4} for $d=1$ and in Lemma~\ref{lem_growth} for $d>1$, the Christoffel-Darboux kernel $K_n^\mu(z,z)$ typically has exponential growth in $\Omega$
 (the unbounded connected component of $\mathbb C^d \setminus \supp(\mu)$) but not in $\supp(\mu)$. This confirms the claim of \cite{Malyshkin1,Malyshkin2,LP1,LP2} that level sets of the Christoffel-Darboux kernel for sufficiently large $n$ allow us to detect {\it outliers}, that is, points of $\supp(\mu)$ which are isolated in $\Omega$.

However, this claim is certainly wrong for atomic measures $\widetilde \nu$ (also referred to as discrete measures)
   $$
        \widetilde \nu = \sum_{j=1}^N t_j \delta_{z^{(j)}}
   $$
   for $t_j >0$ and distinct $z^{(j)}\in \mathbb C^d$ with support describing a cloud of points since, following the above terminology, every element of $\supp(\widetilde \nu)$ would be an outlier, and, even worse, there is only a finite number of orthogonal polynomials making kernel asymptotics impossible. Here typically the point cloud splits into two parts, namely most of the points approximating well some continuous set $S\subset \mathbb C^d$, and some outliers sufficiently far from $S$. To formalize this, we suppose that $N \gg n$, and that the atomic measure $\widetilde \nu$ can be split into two atomic measures
\begin{equation} \label{splitting}
     \widetilde \nu=\widetilde  \mu + \sigma  , \quad
     \widetilde \mu \approx \mu , \quad \supp(\sigma)\subset \Omega,
\end{equation}
with $\mu$ and $\Omega$ as before and $n$ large enough such that we have sufficient information about $K_n^\mu(z,z)$. However, we should insist that in applications we only know the (empirical) measure $\widetilde \nu$ and possibly its moments, and we want to learn from level lines of $K_n^{\widetilde \nu}(z,z)$ the set of outliers $\supp(\sigma)$, and, if possible, $\supp(\mu)$ or even $\mu$ itself.

The notion $\widetilde \mu \approx \mu$ in \eqref{splitting} means that the moments of both measures up to order $n$ are sufficiently close, measured in \S\ref{section_small_perturbations} in terms of modified moments, to insure that $K_n^\mu(z,z)$ is sufficiently close to $K_n^{\widetilde \mu}(z,z)$ for all $z\in \mathbb C^d$. In order to deduce lower and upper bounds for $K_n^{\widetilde \nu}(z,z)$ and actually show that level lines of this Christoffel-Darboux kernel are indeed useful for outlier detection, we still have to examine how a Christoffel-Darboux kernel changes after adding a finite number of point masses represented by $\sigma$, that is, provide lower and upper bounds for
$$
     K_n^{\mu+\sigma}(z,z)/K_n^{\mu}(z,z) \quad\mbox{and} \quad K_n^{\widetilde \mu+\sigma}(z,z)/K_n^{\widetilde \mu}(z,z),
$$ which is the scope of \S\ref{section_additive}. For the special case $d=1$, we establish in \S\ref{section_ratio_asymptotics} asymptotics for $K_n^{\mu+\sigma}(z,z)/K_n^{\mu}(z,z)$ in the special case where we have ratio asymptotics for the underlying $\mu$-orthogonal polynomials. 

Roughly speaking, as long as $N\gg n$, by combining the above results we find two different regimes, namely for $z\in \supp(\sigma)\cup \supp(\mu)$ where $K_n^{\widetilde \nu}(z,z)$ essentially remains bounded, and on compact subsets of the complement $\Omega\setminus \supp(\sigma)$ where $K_n^{\widetilde \nu}(z,z)$ growths exponentially like $K_n^\mu(z,z)$. We may conclude that level lines for $K_n^\nu(z,z)$ for sufficiently large parameters will separate outliers from $\supp(\mu)$, and refer the readers to \cite{LP1,LP2} for a practical choice of the parameter of critical level lines.

We expect the required critical level lines to encircle once either $\supp(\mu)$ or exactly one of the outliers, but not the others. In practical terms, it is non-trivial to draw on a computer certified level lines because of the stiff gradients of $K_n^{\widetilde \nu}(z,z)$, and hence this approach might be quite costly.

Therefore, we suggest a alternate strategy with an overall cost which scales linearly with $N$ (and at most cubic in $n$): Inspired by related results in statistics and data analysis, we are considering in Section~\ref{sec_leverage} a leverage score attached to each mass point $z^{(j)}$, namely
$$
     t_j K_n^{\widetilde \nu}(z^{(j)},z^{(j)}) \in (0,1],
$$
which we expect to be close to $1$ iff $z^{(j)}$ is an outlier. For the particular case of $d=1$ and $\mu$ being area measure on a simply connected and bounded subset of the complex plane, the above-mentioned results on the different Christoffel-Darboux kernels allow us to fully justify this claim, see Corollary~\ref{cor2_Bergman} and the following discussion and numerical experiments.

In this paper we do not discuss and exploit for the moment any recurrence relation between multi-variate orthogonal polynomials and thus the $d$ underlying commuting Hessenberg operators representing the multiplication with the independent variables $z_1,z_2,...,z_d$. This will be considered in a sequel publication.

Adding point masses to a given measure and recording the change of orthogonal polynomials is not a new theme, see for instance the historical notes in \cite[p.~684]{SimonBook}. The latter addresses measures on the unit circle, but the same applies to orthogonality on the real line. In case of several variables, quite a few authors have considered classical measures $\mu$ with explicit formulas for the corresponding orthogonal polynomials like Jacobi-type measures on the unit ball, and $\sigma$ being
one \cite{Lydia} or several mass points \cite{Lydia2}, or the Lebesgue measure on the sphere \cite[Theorem~3.3]{MartinezPinar}.

\subsection{Outline}

\S\ref{section_definition} is devoted to the construction of multivariate orthogonal polynomials and the corresponding Christoffel-Darboux kernel, pointing out in \S\ref{section_def_multivariate} the inherent complications of the multivariate setting: choice of ordering of orthogonal polynomials, degeneracy due to algebraic constraints in the support, lack of positivity characterization of moment matrices, and so on. We also explore in \S\ref{section_properties_multivariate} to what extent some elementary properties of univariate Christoffel-Darboux kernels recalled in \S\ref{section_univariate}(like kernel asymptotics) generalize to the multivariate setting.

In 
\S\ref{section_small_perturbations} we give sufficient conditions insuring that, for fixed $n$, the  Christoffel-Darboux kernels for two different measures behave similarly. This requires us to first address in \S\ref{section_comparing_2_measures} the question of how to compute a family of orthogonal polynomials $p_n^\nu$ from $p_n^\mu$ for two measures $\mu,\nu$, giving rise to so-called modified moment matrices.

Our first main result can be found in \S\ref{section_additive}
where we discuss new upper and lower bounds for the change of the Christoffel-Darboux kernel after adding a finite number of point masses. A second main result given in \S\ref{section_ratio_asymptotics} deals with the special case $d=1$ of univariate polynomials and describes asymptotics for the change of the Christoffel-Darboux kernel after adding a finite number of point masses.
This second result is illustrated in \S\ref{sec_Bergman} by analyzing Bergman orthogonal polynomials. Combining these findings we get a complete analysis of outlier detection in $\mathbb C$ in situation \eqref{splitting} under some smoothness assumptions on $\mu$. We finally illustrate by examples in \S\ref{section_examples_multivariate} that the situation is more involved in the multivariate setting, especially because of the lack of knowledge on asymptotics of some normalized Christoffel-Darboux kernel $C_n^\mu(z,w)$. For the convenience of the reader, an index of notation is included in \S\ref{sec_index}. 
\bigskip \bigskip

\noindent
{\it Acknowledgements.}
The authors are grateful to the Mathematical Research Institute at
Oberwolfach, Germany, which provided exceptional working conditions for a Research in Pairs
collaboration in 2016 during which time the main ideas of this manuscript were developed. Special thanks are due to the
referee for a careful examination of the manuscript and authoritative comments which led to a clarification of Lemma 2.9.

\section{Univariate and multivariate Christoffel-Darboux kernel}\label{section_definition}
 The role of reproducing kernels in the development of the classical theories of moments, orthogonal polynomials, statistical mechanics
 and in general function theory cannot be overestimated. Numerous studies, old and new, recognize the asymptotics of the Christoffel functions
 as the key technical ingredient for a variety of interpolation and approximation questions.

 Our aim is to study the stability of the Christoffel-Darboux kernels under small perturbations or additive perturbations of the underlying measure. Quite a long time ago it was
 recognized that even purely real questions of, say approximation, have to be treated with methods of function theory of a complex variable and potential theory.
 For this reason, we advocate below the complex variable setting. For the advantages and richness of the theory of
 orthogonal polynomials in a single complex variable see
 \cite{Saff,Stahl-Totik}.
\subsection{Definition and basic properties in the univariate case}\label{section_univariate}
Consider the Lebesgue space $L^2(\mu)$ with scalar product and norm
\begin{equation} \label{def_norm}
    \langle f,g \rangle_{2,\mu} = \int f(z) \ol{g(z)} d\mu(z), \quad
    \| f \|_{2,\mu} = \Bigl( \langle f,f \rangle_{2,\mu} \Bigr)^{1/2},
\end{equation}
where $\mu$ is a positive Borel measure in the complex plane.
Assume that the measure $\mu$ has compact and infinite support $\supp(\mu)$, so that one can orthogonalize without ambiguity the sequence of monomials.
Then the Gram-Schmidt process gives us the associated sequence of orthonormal polynomials, $p_j^\mu$ of degree $j$, with positive leading coefficient, together with the Christoffel-Darboux kernel
\begin{equation} \label{def_kernel1}
    K^\mu_n(z,w) = \sum_{j=0}^n p^\mu_j(z) \ol{p^\mu_j(w)}
\end{equation}
for the subspace of polynomials of degree at most $n$.
 The case $z=w$ will be referred to as the {\it diagonal kernel}, whereas for $z\neq w$ we speak of a \textit{polarized kernel}.

For any $z\in \C$ and any polynomial $p$ of degree at most $n$ we have the reproducing property
\begin{equation} \label{def_kernel2}
     \langle p,K_n^\mu(\cdot,z)\rangle_{2,\mu} = \int K_n(z,w) p(w) d\mu(w) = p(z);
\end{equation}
in other words, the function $w \mapsto K_n^\mu(w,z)$ represents the linear functional of point evaluation $p\mapsto p(z)$ in the algebra of polynomials. Simple Hilbert space arguments show that $K_n^\mu(z,w) = \langle K_n^\mu(\cdot,w),K_n^\mu(\cdot,z)\rangle_{2,\mu}$, and that, for all $z\in \mathbb C$,
\begin{equation} \label{def_kernel3}
        \frac{1}{\sqrt{K_n^\mu(z,z)}} = \min \{ \| p \|_{2,\mu} : \mbox{~$p$ polynomial of degree $\leq n$, $p(z)=1$} \},
\end{equation}
with the unique extremal polynomial being given by $x\mapsto K_n^\mu(x,z)/K_n^\mu(z,z)$.

The reciprocal $\lambda_n(z) = 1/{K_n^\mu(z,z)}$ naturally appears in a variety of approximation questions in the complex plane, and is traditionally called the
{\it Christoffel} or {\it Christoffel-Darboux} function.
For instance, it is shown by Totik \cite[Theorem~1.2]{To09} that, provided that $\mu$ is sufficiently smooth and supported on a system of arcs, the quantity $n/K_n(z,z)$ tends to the Radon-Nykodym derivative of $\mu$ on open subarcs. If $z$ belongs to the (two-dimensional) interior of $\supp(\mu)$, a similar result is established for $n^2/K_n^\mu(z,z)$ in \cite[Theorem~1.4]{To09}. Under weaker assumptions we may also describe $n$th root limits of the Christoffel-Darboux kernel: Denote by $\Omega$ the unbounded connected component
   of $\mathbb C\setminus \supp(\mu)$, with $g_\Omega(\cdot)$ its Green function\footnote{Recall that $g_\Omega(z)>0$ for $z\in \Omega$, and equal to infinity if $\supp(\mu)$ is polar, that is, of logarithmic capacity zero.}  with pole at $\infty$.
Using the inequality
   $$
        \| p \|_{2,\mu}^2 \leq \mu(\mathbb C)\, \| p \|_{L^\infty(\supp(\mu))}^2
   $$
in \eqref{def_kernel3} and, e.g., the extremal property \cite[Theorem~III.1.3]{Saff-T} of Fekete points, we get at least exponential growth for the Christoffel-Darboux kernel in $\Omega$, namely
\begin{equation} \label{def_kernel4}
   \liminf_{n\to \infty} K_n^\mu(z,z)^{1/n} \geq e^{2g_{\Omega}(z)}
\end{equation}
uniformly on compact subsets of $\Omega$. Moreover, if $\Omega$ is regular with respect to the Dirichlet problem, then for sufficiently ``dense" measures, namely the class {\bf Reg} introduced in \cite[\S 3]{Stahl-Totik}, it is known by \cite[Theorem~3.2.3]{Stahl-Totik} that
\begin{equation} \label{def_kernel4'}
   \lim_{n\to \infty} K_n^\mu(z,z)^{1/n} = e^{2g_{\Omega}(z)}
\end{equation}
uniformly on compact subsets of $\mathbb C$. In particular, there is no exponential growth on the support. Roughly saying, this is the motivation of \cite{LP1,LP2} for considering level sets of the Christoffel-Darboux kernel to  help detect outliers, that is, points of the support of $\mu$ which are isolated in $\Omega$.

\subsection{Definition of the Christoffel-Darboux kernel in the multivariate case}\label{section_def_multivariate}

The theory of orthogonal polynomials of several (real or complex) variables is much more involved and thus less developed, we refer the reader for instance to \cite{DX} or the recent summary \cite{Xu_recent} for the case of several real variables. Briefly speaking, we have to face at least two basic problems which we will address in the next two subsections. The first problem is that there is no natural ordering of multivariate monomials. Second, the multivariate monomials might not be linearly independent in $L^2(\mu)$, or, in other words, the null space
$$
     \mathcal N(\mu)=\{ \mbox{$p \in \C[z]$}: \int |p|^2 d\mu = 0 \}
$$
might be nontrivial. This means in particular that we have to revisit the reproducing property \eqref{def_kernel2}, and the extremal problems \eqref{def_kernel3} and \eqref{def_kernel4}, the aim of \S\ref{section_properties_multivariate}.

Henceforth $\C[z]$ stands for the algebra of polynomials with complex coefficients in the indeterminates $z = (z_1,z_2,\ldots,z_d)$. Notice that every function continuous in $\mathbb C^d$ and vanishing $\mu$-almost everywhere has to vanish on $\supp(\mu)$, which actually shows that $\mathcal N(\mu)$ does only depend on $\supp(\mu):$
$$
     \mathcal N(\mu)=\{ \mbox{$p \in \C[z]$: $p$ vanishes on $\supp(\mu)$} \}.
$$
It immediately follows that
${\mathcal N}(\mu)$ is a (radical) ideal in the algebra $\mathbb C[z]$.

 Hereafter, $\mu$ denotes a positive Borel measure with compact support in $\C^d$, with corresponding scalar product as in \eqref{def_norm}. Depending on
 the context, the hermitian inner product in $\C^d$ is denoted as follows:
 $$ \langle z, w \rangle = z \cdot \overline{w} = w^\ast z = z_1 \overline{w}_1 + z_2 \overline{w}_2 + \cdots + z_d \overline{w}_d.$$
 In a few instances we will also encounter the bilinear form
 $$ z \cdot w = z_1 {w}_1 + z_2 {w}_2 + \cdots + z_d {w}_d.$$

 The standard multi-index notation
 $$ z^\alpha = z_1^{\alpha_1} z_2^{\alpha_2} \cdots z_d^{\alpha_d}, \ \  \alpha = (\alpha_1,\alpha_2,\ldots,\alpha_d) \in \N^d,$$
 is adopted. The graded lexicographical order on the set of indices is denoted by $\lq\lq\leq_\ell"$. Specifically
 $$ \alpha \leq_\ell \beta$$ if either $\alpha=\beta$ or $\alpha<\beta$; i.e.,
 $$ |\alpha| = \alpha_1 + \cdots +\alpha_d < \beta_1+\cdots +\beta_d = |\beta|,$$
 or
 $$ |\alpha| = |\beta| \ \ {\rm and} \ \ \alpha_1 = \beta_1, \ldots, \alpha_k = \beta_k, \alpha_{k+1} < \beta_{k+1},$$
 for some $k<d$.

 In what follows we suppose that a fixed enumeration $\alpha(0),\alpha(1),\alpha(2),...$ of all multi-indices is given with $\alpha(0)=0$, for instance (but not necessarily) a graded lexicographical ordering.
Henceforth we assume that multiplication by any variable increases the order;
\footnote{
   This is equivalent to requiring that any subset of multi-indices $\{ \alpha(0),\alpha(1),\alpha(2),..., \alpha(n) \}$ is downward close, compare with, e.g., \cite{cohen}.
}
that is, for all $j,n$ that $z_j z^{\alpha(n)}=z^{\alpha(k)}$ for $k>n$.
 We write $\deg p = k$ if $p$ is a linear combination of the monomials $z^{\alpha(0)},z^{\alpha(1)},...,z^{\alpha(k)}$, with the coefficient of $z^{\alpha(k)}$ being non-trivial. Note that, for $d>1$, our notion of degree differs from the so-called total degree $\tdeg p$ being the largest of the orders $|\alpha|$ of monomials $z^\alpha$ occurring in $p$.

\begin{example}\label{example_dependencies}
  Consider the positive measure $\mu$ on $\C^2$ given by
$$ \int \overline{p(z_1,z_2)} q(z_1,z_2) d\mu(z)
= \int_{|u|\leq 1} \overline{p(u^3,u^2)} q(u^3,u^2) dA(u),$$
where $p,q\in \mathbb C[z]$ and $dA$ stands for the area measure in $\mathbb C$. Notice that $\supp(\mu)$ is compact and strictly contained in the complex affine variety $$
    \{ (z_1,z_2)\in \mathbb C: P(z_1,z_2)=0\}; \quad
    P(z_1,z_2) := z_1^2 - z_2^3 ,
$$
in particular, the ideal $\mathcal N(\mu)$ is generated by $P$. The graded lexicographical ordering for bivariate monomials is given by
$$
        1, z_1,z_2,z_1^2,z_1z_2,z_2^2,z_1^3,z_1^2z_2,z_1z_2^2,z_2^3,z_1^4,z_1^3z_2,z_1^2z_2^2,z_1z_2^3,z_2^4,...
$$
and, for instance, $\alpha(8)=(2,1)$. Notice that, seen as functions restricted to the above affine variety, the two monomials $z_2^3$ and $z_1^2$ cannot be distinguished, and the same is true for $z_1z_2^3$ and $z_1^4$, and for $z_2^4$ and $z_1^3z_2$, and so on. We thus have to go to the quotient space $\mathbb C[z_1,z_2]/\mathcal N(\mu)$, with a basis given by
$$
        1, z_1,z_2,z_1^2,z_1z_2,z_2^2,z_1^3,z_1^2z_2,z_1z_2^2,z_1^4,...,z_1^n,z_1^{n-1}z_2, z_1^{n-2}z_2^2,z_1^{n+1},...,
$$
written in the same order. We will see in Lemma~\ref{lem_degree_OP} below, that exactly this maximal set of monomials being linearly independent in $L^2(\mu)$ will allow us to construct corresponding multivariate orthogonal polynomials.

\end{example}

\begin{defi}
    Let $k \in \N$. The {\it tautological vector} of ordered monomials is the row vector with $(k+1)$ components, given by
    $$ v_k(z)  = (z^{\alpha(j)})_{j=0,1,...,k}.$$
    The corresponding matrix of moments is defined by
 $$
     \widetilde M_k(\mu)= \Bigl( \langle z^{\alpha(\ell)},z^{\alpha(j)} \rangle_{2,\mu} \Bigr)^k_{j,\ell=0}=
     \int v_k(z)^\ast v_k(z) \, d\mu(z) .
 $$
\end{defi}
 Notice that $\widetilde M_k(\mu)$ is a Gram matrix and hence positive semi-definite. However, these moment matrices are not necessarily of full rank $k+1$. This may happen already in the case $d=1$ of univariate orthogonal polynomials, where $z^{\alpha(j)}=z^j$ and the support of $\mu$ is finite. Indeed, for $d=1$,
 assume that $\det \widetilde M_n(\mu)=0$ and $n$ is the first index with this property. That means there exists a non-trivial combination of the rows of the matrix $\widetilde M_n(\mu)$, or in other terms,
 there exists a monic polynomial $h(z) = z^n + \cdots $ of exact degree $n$, such that $h \in \mathcal N(\mu)$.
 That is, $\mu$ is an atomic measure with exactly $n$ distinct points in its support, namely the roots of $h$. Consequently, $h(z)z^j$ is also annihilated by $\mu$ for all $j \geq 0$, showing that $\mathcal N(\mu)$ is the ideal generated by $h$, and
 $$ {\rm rank} \widetilde M_{n+j}(\mu) = {\rm rank} \widetilde M_n(\mu) = n, \ \  j \geq 0.$$
 In case of several variables, the ideal $\mathcal N(\mu)$ is not necessarily a principal ideal, but it is finitely generated, according to Hilbert's basis theorem.  However, in general the above rank stabilization phenomenon of moment matrices ceases to hold for $d>1$.

 In the case $d \geq 1$, we may still generate orthogonal polynomials $p_0^\mu, p_1^\mu,...$ by the Gram-Schmidt procedure starting with the family of functions $z^{\alpha(0)}$, $z^{\alpha(1)}, ...$; however, it may happen that not every monomial $z^{\alpha(k)}$ gives raise to a new orthogonal polynomial. Indeed, since we have fixed an ordering of the monomials, we may and will construct a maximal set of monomials $z^{\alpha(k_n^\mu)}$, with $0=k_0^\mu<k_1^\mu<\cdots$ being linearly independent in $L^2(\mu)$, and an orthogonal basis $p_0^\mu, p_1^\mu,...$ of the set spanned by these monomials. As we will see in Lemma~\ref{lem_decomposition} below, this set together with $\mathcal N(\mu)$ will allow us to write the set of multivariate polynomials as an orthogonal sum.

 The next result, necessarily technical, formalizes the construction of orthogonal polynomials in the presence of an ordering of the monomials and possible linear dependence of the
 associated moments. The interested reader might first have a look at the special case $\mathcal N(\mu)=\{ 0 \}$ where the moment matrices $\widetilde M_k(\mu)$ have full rank $k+1$ for all $k\in \mathbb N$, and thus orthogonal polynomials exist for each degree $n=k_n^\mu$.

 \begin{lem}\label{lem_degree_OP}
    We may construct multivariate orthogonal polynomials $p_0^\mu, p_1^\mu,...$ with
    \begin{equation} \label{eq.degree_OP1}
       \deg p_n^\mu = k_n^\mu := \min \{ k \geq 0 : {\rm rank} \widetilde M_k(\mu)=n+1 \};
    \end{equation}
    in particular, $k_0^\mu=0$.
    Moreover, for $k=k_n^\mu$, there exists
    an invertible and upper triangular matrix $\widetilde R_k(\mu)$ of order $k+1$ such that
    \begin{equation} \label{eq.degree_OP2}
       \widetilde R_k(\mu)^\ast \widetilde M_k(\mu)\widetilde R_k(\mu)
       = E_k(\mu) E_k(\mu)^\ast , \quad E_k(\mu):=(\delta_{j,k_\ell^\mu})_{j=0,...,k,\ell=0,...,n}.
    \end{equation}
    Furthermore, $M_n(\mu)=E_k(\mu)^\ast \widetilde M_k(\mu) E_k(\mu)$ is an maximal invertible submatrix of $\widetilde M_k(\mu)$, and
    \begin{equation} \label{eq.degree_OP3}
       R_n(\mu)^\ast M_n(\mu)R_n(\mu) = I , \quad v_n^\mu:=(p_0^\mu,...,p_n^\mu) = v_k E_k(\mu) R_n(\mu),
    \end{equation}
    with the upper triangular and invertible matrix $R_n(\mu)=E_k(\mu)^\ast \widetilde R_k(\mu) E_k(\mu)$.
 \end{lem}
 \begin{proof}
  We construct orthogonal polynomials with \eqref{eq.degree_OP1} by recurrence. Since $\| 1 \|_{2,\mu}^2=\mu(\mathbb C^d) \neq 0$, we may define
   $q_0^\mu(z)=1$ and $p_0^\mu(z)=1/\sqrt{\mu(\mathbb C^d)}$.
   At step $k\geq 1$, with $k_{n-1}^\mu< k \leq k_{n}^\mu$, we compute $q_k^\mu=v_k \, \xi$ by orthogonalizing $z^{\alpha(k)}$ against $p_0^\mu,...,p_{n-1}^\mu$. By recurrence hypothesis \eqref{eq.degree_OP1} and the assumption $k_{n-1}^\mu<k$, the degree of any of the polynomials $p_0^\mu,...,p_{n-1}^\mu$ is strictly less than $k$. Thus $\deg q_k^\mu=k$; that is, the last component of $\xi$ is nontrivial. If $k=k_n^\mu,$ then by definition of $k_n^\mu$ the last column of $\widetilde M_k(\mu)$ is linearly independent of the others, and hence $\widetilde M_k(\mu) \xi \neq 0$. Using the fact that $\widetilde M_k(\mu)$ is Hermitian and positive semi-definite, we infer
   $$
     \| q_k^\mu \|_{2,\mu}^2=\xi^\ast \widetilde M_k(\mu) \xi \neq 0.
   $$
Consequently one can define
   $p_n^\mu=q_k^\mu/\| q_k^\mu \|_{2,\mu}$ of degree $k_n^\mu$, with positive leading coefficient. If however $k<k_n^\mu$, we find that $\widetilde M_k(\mu) \xi = 0$ and thus $\| q_k^\mu \|_{2,\mu}^2=0$, in other words,
   $q_k^\mu \in \mathcal N(\mu)$. The same argument shows that there is no polynomial of degree exactly $k$ and of norm $1$ that is orthogonal to $p_0^\mu,...,p_{n-1}^\mu$. This proves \eqref{eq.degree_OP1}.

   Redefining $q_k^\mu(z)=p_n^\mu(z)$ for all $k=k_n^\mu$, we have thus constructed a basis $\{ q_0^\mu,...,q_k^\mu \}$ of the space of polynomials of degree at most $k$; in other words, there exists an upper triangular and invertible matrix $\widetilde R_k(\mu)$ with $$
     (q_0^\mu(z),...,q_k^\mu(z))=v_k(z)\widetilde R_k(\mu),
   $$
   and \eqref{eq.degree_OP2} follows by orthonormality.

   By construction, $p_n^\mu$ contains only the powers $$
       z^{\alpha(k_0^\mu)},...,z^{\alpha(k_n^\mu)};
   $$
   that is, the columns of the coefficient vectors of our orthogonal polynomials satisfy
   $$
          \widetilde R_k(\mu)E_n(\mu)  = E_n(\mu)E_n(\mu)^\ast \widetilde R_k(\mu)E_n(\mu) = E_n(\mu) R_n(\mu),
   $$
   which implies \eqref{eq.degree_OP3}.
 \end{proof}

 Notice that \eqref{eq.degree_OP2} gives a full rank decomposition $\widetilde M_k(\mu)=X X^\ast$, with the factor $X^\ast=E_n(\mu)^\ast \widetilde R_n(\mu)^{-1}$ being of upper echelon form, the pivot in the $j$-th column lying in row $k_j^\mu$. We also learn from \eqref{eq.degree_OP2} that a basis of the nullspace of $\widetilde M_k(\mu)$ is given by the columns of $\widetilde R_k(\mu)$ with indices different from $k_0^\mu,k_1^\mu,...$. Finally, from \eqref{eq.degree_OP2}, $R_n(\mu)^{-1}$ is the Cholesky factor of the invertible matrix $M_n(\mu)$.

 With the vector of orthogonal polynomials $v_n^\mu$ as in Lemma~\ref{lem_degree_OP}, we are now prepared to define the {\it multivariate Christoffel-Darboux kernel} of order $n$ as
  $$ K^\mu_n (z,w) = v^\mu_n(z) v^\mu_n(w)^\ast = \sum_{j=0}^n p^\mu_j(z) \overline{p^\mu_j(w)}, $$
 where we notice that, according to \eqref{eq.degree_OP3},
 $$
      K^\mu_n (z,w) = v_{k_n^\mu}(z) E_n(\mu) M_n(\mu)^{-1} E_n(\mu)^\ast v_{k_n^\mu}(w)^\ast.
 $$

For later use we mention the affine invariance of orthogonal polynomials and Christoffel functions : consider the change of variables $\widetilde z=B z+b$ for a diagonal and invertible $B\in \mathbb C^{d\times d}$ and $b\in \mathbb C^d$, and let $c>0$. If we define $\widetilde \mu$ by $\widetilde \mu(S)=c \mu(B S+b)$ for any Borel set $S\subset \mathbb C^d,$ then
\begin{equation} \label{affine_invariance}
       K_n^\mu(z,w) = c\, K_n^{\widetilde \mu}(Bz+b,Bw+b) ,
\end{equation}
compare with \cite[Lemma~1]{LP1} where the authors allow also for general invertible matrices $B$ but then require graded lexicographical ordering and the additional assumption $\mathcal N(\mu)=\{ 0 \}$.
\subsection{Basic properties of multivariate Christoffel-Darboux kernels}\label{section_properties_multivariate}

\begin{defi}\label{def_sets}
Consider for $k=k_n^\mu$ the following two subsets of multivariate polynomials
\begin{eqnarray*}
    &&\mathcal N_n(\mu):=\{ v_{k} \xi : \widetilde M_{k}(\mu) \xi = 0 \}
    = \{ p\in \mathcal N(\mu) : \deg p \leq k \},
    \\
    &&\mathcal L_n(\mu):=\spann \{ p_j^\mu : j=0,1,...,n \} = \spann \{ z^{\alpha(k_j^\mu)} : j=0,1,...,n \},
\end{eqnarray*}
    and the affine variety
    \begin{eqnarray*}
        \mathcal S_n(\mu) := \bigcap_{p \in  \mathcal N_n(\mu)} \{ z \in \mathbb C^d : p(z)=0 \}.
    \end{eqnarray*}
\end{defi}

By definition of $\mathcal N_n(\mu)$, we get $\supp(\mu)\subset \mathcal S_n(\mu)$.  Notice that $\mathcal N_n(\mu)$ is increasing in $n$ and hence $\mathcal S_n(\mu)$ decreases in $n$. Moreover, recalling that $\mathcal N(\mu)=\bigcup_{n\geq 0} \mathcal N_n(\mu)$ is a finitely generated ideal, we see that, for $n$ being larger than or equal to the largest of the degrees of the generators, the set $\mathcal S_n(\mu)$ does no longer depend on $n$, and we will use the notation $\mathcal{S}(\mu)$, being the {\it Zarisky closure} of $\supp(\mu)$.

Notice that $\mathcal N(\mu)=\{ 0 \}$ implies that $\mathcal{S}(\mu)=\mathbb{C}^d$. In case of an atomic measure $\mu$ with a finite number of point masses one can show that $\mathcal{S}(\mu)=\supp(\mu)$, but in general $\mathcal{S}(\mu)$ is larger than $\supp(\mu)$; see for instance Example~\ref{example_dependencies}.

\begin{examples}\label{exa_supp}
   Consider two measures $\mu,\nu$, with $\supp(\mu)\subset \supp(\nu)$, subject for instance to the condition $\nu-\mu\geq 0$. Then obviously $\mathcal N(\nu)\subset \mathcal N(\mu)$, with equality if and only if $\supp(\nu)\subset \mathcal{S}(\mu)$. The Lebesgue measure $\mu$ on the real or complex unit ball in $\mathbb C^d$ has been extensively studied; in particular $\mathcal N(\mu)=\{ 0 \}$, see \S\ref{section_examples_multivariate}. We conclude that, if $\supp(\nu)$ has a non-empty interior in $\mathbb C^d$ (or $\supp(\mu)\subset \mathbb R^d$ has a non-empty interior in $\mathbb R^d$), then $\mathcal N(\nu)=\{ 0 \}$ and $\mathcal{S}(\nu)=\mathbb C^d$.
\end{examples}

It is easy to check that we keep the reproducing property \eqref{def_kernel2} for $p\in \mathcal L_n(\mu)$, but it fails for nontrivial elements of $\mathcal N_n(\mu)$. The following result shows that the set of multivariate polynomials of degree at most $k=k_n^\mu$ can be written as an orthogonal sum $\mathcal L_n(\mu) \oplus \mathcal N_n(\mu)$. Thus the multivariate Christoffel-Darboux kernel is indeed a projection kernel onto $\mathcal L_n(\mu)$, and only a reproducing kernel in the case $\mathcal N_n(\mu)=\{ 0 \}$.

\begin{lem}\label{lem_decomposition}
        Let $k,n\in \mathbb N$ with $k_n^\mu \leq k < k_{n+1}^\mu$, and $p$ be a polynomial of degree $\leq k$. Then the polynomial defined by
        $$
            q(z):=\int K_n^\mu(z,w)p(w) d\mu(w)
        $$ satisfies $q \in \mathcal L_n(\mu)$ and $p-q \in \mathcal N_{n+1}(\mu)$. In particular, for $k=k_n^\mu=\deg p_n^\mu$ we have that
        \begin{equation} \label{eq.decomposition}
              \mathcal N_n(\mu) = \{ p \in \C[z]: \deg p \leq \deg p_n^\mu , \  p \perp_\mu p_0^\mu,...,p_n^\mu \} .
        \end{equation}
\end{lem}
    \begin{proof}
      By construction, $q\in \mathcal L_n(\mu)$. Also, $\deg (p-q)\leq k < \deg p_{n+1}^\mu$, and, by orthogonality,
      \begin{eqnarray*} &&
          \int |p-q|^2 d\mu = \| p \|_{2,\mu}^2 - \| q \|_{2,\mu}^2  = \| p \|_{2,\mu}^2 - \sum_{j=0}^n | \langle p, p_j^\mu \rangle_{2,\mu} |^2 = 0,
      \end{eqnarray*}
      showing that $p-q\in \mathcal N_{n+1}(\mu)$, as claimed in the first assertion. Notice that in the special case $k=k_n^\mu$, the same reasoning gives the slightly more precise information
      $p-q\in \mathcal N_{n}(\mu)$, and $p\in \mathcal N_n(\mu)$ iff $q=0$ iff $p\perp_\mu p_0^\mu,...,p_n^\mu$, leading to the second part.
    \end{proof}

 Let us now investigate to what extent the extremal property \eqref{def_kernel3} remains true in the multivariate setting.
 If $z\not\in \mathcal S_n(\mu),$ then by definition there exists $p\in \mathcal N_n(\mu)$ with $p(z)=1$, and hence $\inf \{ \| p \|_{2,\mu} : \deg p \leq k_n^\mu, \ p(z)=1 \}=0$, showing that \eqref{def_kernel3} fails. There are two ways to specialize this extremal problem in order to recover \eqref{def_kernel3}.

\begin{lem}
  For all $n \in \mathbb N$ and $z\in \mathbb C^d$
 \begin{eqnarray} \label{def_kernel3_bis}
       &&
        \frac{1}{\sqrt{K_n^\mu(z,z)}} = \min \{ \| p \|_{2,\mu} :  p \in \mathcal L_n(\mu), p(z)=1 \} ,
 \end{eqnarray}
  and for all $k,n\in \mathbb N$ with $k_n^\mu \leq k < k_{n+1}^\mu$ and $z\in \mathcal S_{n+1}(\mu)$
 \begin{eqnarray}
      \label{def_kernel3_ter}
       &&
        \frac{1}{\sqrt{K_n^\mu(z,z)}} = \min \{ \| p \|_{2,\mu} :  \deg p \leq k, p(z)=1 \} .
 \end{eqnarray}
\end{lem}
\begin{proof}
  For a proof of \eqref{def_kernel3_bis}, we use exactly the same argument as in the univariate case, namely the Cauchy-Schwarz inequality and its sharpness. Let now be  $z\in \mathcal S_{n+1}(\mu)$. Then, with $p$ and $q\in \mathcal L_n(\mu)$ as in Lemma~\ref{lem_decomposition}, we have that $\| p \|_{2,\mu}=\| q \|_{2,\mu}$, and $p(z)-q(z)=0$ by definition of $\mathcal S_{n+1}(\mu)$. This shows that in \eqref{def_kernel3_ter} we may restrict ourselves to $p\in \mathcal L_n(\mu)$, and \eqref{def_kernel3_ter} follows from \eqref{def_kernel3_bis}.
\end{proof}

We learn from \eqref{def_kernel3_ter} that, for $z\in \mathcal S(\mu)$,  the Christoffel-Darboux kernel $K_n^\mu(z,z)$ does not depend on the particular ordering of monomials as long as the set $\{ \alpha(0),\alpha(1),...,\alpha(k)\}$ remains the same.

 Of related interest will be the {\it cosine function}
 \begin{equation}\label{cos}
 C^\mu_n(z,w) := \frac{K^\mu_n (z,w)}{\sqrt{K^\mu_n (z,z)}\sqrt{K^\mu_n (w,w)}}.
 \end{equation}
 To the best of our knowledge, there are hardly any results in the literature concerning asymptotics for $n\to \infty$ for such
  a normalized and polarized Christoffel-Darboux kernel.
 Multipoint matricial analogs of these kernels will be used later on.
\begin{defi}\label{def_multi_kernel}
 Let $z_1, z_2, \ldots, z_\ell, w_1,w_2,\ldots, w_\ell$ be two $\ell$-tuples of arbitrary points in $\C^d$, and define the $\ell\times \ell$ matrices
 \begin{eqnarray*} &&
     K^\mu_n (z_1, z_2, \ldots, z_\ell; w_1,w_2,\ldots, w_\ell) := ( K^\mu_n(z_j, w_k))_{j,k=1}^\ell ,
     \\&&
     C^\mu_n (z_1, z_2, \ldots, z_\ell; w_1,w_2,\ldots, w_\ell) := ( C^\mu_n(z_j, w_k))_{j,k=1}^\ell.
\end{eqnarray*}
\end{defi}
 If $z_1,...,z_n\in \mathcal{S}_n(\mu)$ are distinct and $K^\mu_n (z_1, z_2, \ldots, z_\ell; z_1,z_2,\ldots, z_\ell)$ is invertible (compare with Lemma~\ref{lem_K_invertible} below), then this matrix also occurs in variational problems similar to the single point case. Specifically
 \begin{eqnarray*} &&
    \min \{ \| p \|_{2,\mu}^2 : \deg p \leq k_n^\mu,\  p(z_j)=c_j \mbox{~for~} j=1,...,\ell \}
    \\&& =
    c^\ast  K^\mu_n (z_1, z_2, \ldots, z_\ell; z_1,z_2,\ldots, z_\ell)^{-1} c, \quad c^T=(c_1,\ldots,c_\ell).
 \end{eqnarray*}
 This follows along the same lines as the proof of \eqref{def_kernel3_ter} by observing that it is sufficient to examine candidates of the form
 $p(z) = a_1 K_n^\mu(z,z_1)+\cdots+a_\ell K_n^\mu(z,z_\ell)$ with $p(z_j)=c_j$ for $j=1,...,\ell$.

When trying to extend the asymptotic properties \eqref{def_kernel4} and \eqref{def_kernel4'} to
the multivariate setting one encounters a series of complications.
First of all, the sought limiting behavior of Christoffel-Darboux kernels
{\it outside} the support of the generating measure is hardly discussed in the literature, especially in the general setting of \S\ref{section_def_multivariate}. In Lemma~\ref{lem_growth} below we will therefore restrict ourselves to measures $\mu$ with support being non-pluripolar, 
 which allows to apply pluripotential theory, but excludes cases as in
 Example~\ref{example_dependencies} where the support is an algebraic variety.
 As for many other asymptotic results in the literature, we also have to restrict ourselves to total degree, and therefore assume that,
 for all $m\in \mathbb N$, we have that  
 \begin{equation} \label{total_degree}
     \{ \alpha(0),...,\alpha(n) \} = \{ \alpha \in \mathbb N^d : | \alpha|\leq m \} ,
 \end{equation} with $n=:n_{tot}(m)=\Bigl( {{m+d}\atop d} \Bigr)-1$, as it is, for instance, true for the graded lexicographical ordering.
 In addition, we are forced to focus only on the complement of the polynomial hull of $\supp(\mu)$.

Lemma~\ref{lem_growth} below summarizes some basic results of pluripotential theory.
Without aiming at the most comprehensive (and hence more technical) results,
we confine ourselves at offering a multivariate analogue to the much better understood
case of a single complex variable. We stress that Lemma~\ref{lem_growth} below is not used anywhere else in the present article. Our main focus is perturbation formulas for the Christoffel-Darboux kernel and
sharp asymptotics. As discussed below such formulas are of course desirable, but we have to be well aware of the
constraints imposed by their existence.

\begin{lem}\label{lem_growth}
   Consider an ordering with \eqref{total_degree}, and suppose that $S:=\supp(\mu)$ is non-pluripolar. Denote by $\Omega$ be the unbounded connected component of $\mathbb C^d \setminus S$, and by $\widehat S$ the polynomial convex hull\footnote{By definition \cite[p.~75]{Klimek}, $\widehat S=\{ z\in \mathbb C^d: |p(z)|\leq \| p \|_{L^\infty(S)} \mbox{~for all polynomials~} p \}$. One easily checks that $\widehat S$ is compact and contains $S$. Moreover, 
   $S=\widehat S$ for real $S$ by \cite[Lemma~5.4.1]{Klimek}.  However, unlike the case of one complex variable, $\widehat S$ might be strictly larger than the complement of $\Omega$, see, e.g., \cite{AlexanderWermer}.} of $S$.
  \\
   {\bf (a)} For $z\in \Omega\setminus \widehat S$,
   \begin{equation} \label{def_kernel4_bis}
        \liminf_{m \to \infty} K_{n_{tot}(m)}^\mu(z,z)^{1/m} > 1 . 
   \end{equation}
   \\
   {\bf (b)} Suppose, in addition, that $S$ is $L$-regular in the sense of \cite[p.186 and \S 5.3]{Klimek}, and denote by $g_{\Omega}$ the plurisubharmonic Green function with pole at $\infty$ which is plurisubharmonic and continuous in $\mathbb C^d$, equals $0$ in $\widehat S$ and is strictly positive in $\mathbb C^d \setminus \widehat S$; see the proof below. Then, for all $z\in \mathbb C^d$,
   \begin{equation} \label{def_kernel4_bis'}
        \liminf_{m \to \infty} K_{n_{tot}(m)}^\mu(z,z)^{1/m} \geq e^{2g_\Omega(z)} .
   \end{equation}
   \\
   {\bf (c)} If, moreover, $\mu$ belongs to the multivariate generalization of the class {\bf Reg} introduced in \cite[Eqn.\ (1.6)]{Kroo-Lubinsky-1} (or, equivalently, $(S,\mu)$ satisfies a Bernstein-Markov property in the sense of \cite[Eqn.~(2)]{BloomLevenbergPiazzonWielonsky}), then the limit
   \begin{equation} \label{def_kernel4_ter}
        \lim_{m \to \infty} K_{n_{tot}(m)}^\mu(z,z)^{1/m}
                = e^{2g_\Omega(z)}
   \end{equation}
   exists for all $z\in \mathbb C^d$, and equals one
   iff $z\in \widehat S$.
\end{lem}
\begin{proof}
   Our proof of Lemma~\ref{lem_growth} is based on Siciak's function $\Phi_S$ in $\mathbb C^d$ for the compact set $S=\supp(\mu)$:
   \begin{eqnarray*}
       \Phi_{S}(z) &:=& \sup \left\{ \left( \frac{|p(z)|}{\| p \|_{L^\infty(S)}} \right)^{1/\tdeg(p)} :  p\neq 0 \mbox{~a polynomial} \right\}
       \\&=&
       \lim_{m \to \infty} \sup \left\{ \left( \frac{|p(z)|}{\| p \|_{L^\infty(S)}} \right)^{1/m}  : p \neq 0, \tdeg(p)\leq m
       \right\},
   \end{eqnarray*}
   where we recall that $\tdeg(p)$ is the largest of the orders $|\alpha|$ of monomials $z^\alpha$ occurring in $p$. Taking a fixed polynomial $p$ in the second formula shows in particular that $\Phi_{S}(z)\geq 1$ for $z\in \mathbb C^d$.
   By definition of the polynomial convex hull, we have that $\Phi_S(z)\leq 1$ for $z\in \widehat S$, and thus $\Phi_S(z)=1$ for $z\in \widehat S$. Moreover, if $z\not\in \widehat S$, then there exists a polynomial $q$ with $|q(z)|/\| q \|_{L^\infty(S)} >1$, and taking as $p$ suitable powers of $q$ we may conclude that $\Phi_S(z)>1$.

   Next we invoke and apply several arguments from pluripotential theory; for a summary we refer the reader to \cite{BloomLevenberg,Plesniak2003,Sad} or \cite[Appendix B]{Saff-T} as well as the monograph \cite{Klimek}.
   Denote by $\mathcal L$ the Lelong class of plurisubharmonic functions having at most logarithmic growth at $\infty$, and consider the so-called pluricomplex Green function $V_S(z):=\sup \{ u(z): u \in \mathcal L, u|_S \leq 0\}$, together with its upper semi-continuous regularization $$
       V_{S}^*(z) = \limsup_{w\to z} V_{S}(w).
   $$
   It is shown, e.g., in \cite[Theorem~2]{Sad} or \cite[Theorem~B.2.8]{Saff-T}, that $\log \Phi_S(z)=V_S(z)$ for all $z\in \mathbb C^d$. Hence from the above-mentioned properties of Siciak's function we may conclude that $V_S=0$ in $\widehat S$, and $V_S>0$ in $\mathbb C^d \setminus \widehat S$.
   By \cite[Definition~B.1.6 and Theorem~B.1.8]{Saff-T}, the function $V_{S}^*$ is plurisubharmonic whereas, in general, $V_{S}$ is not. By definition, $V_S^*\geq 0$ in $\mathbb C^d$, $V_S^*>0$ outside $\widehat S$, and $V_S^*=0$ on $\widehat S$ up to some pluripolar set by \cite[Theorem~B.1.7]{Saff-T}. For all these reasons, following \cite{Sad} we write $$
       g_\Omega(z):=V_{S}^*(z)
   $$ referred to as the plurisubharmonic Green function of $S$ with pole at infinity since the latter function reduces to the classical univariate Green function in case $d=1$. Notice also that $g_\Omega=g_{\mathbb C^d\setminus \widehat S}$ by \cite[Corollary~5.1.7]{Klimek}.
   For the particular case of $S$ being $L$-regular, we know that $V_S$ is continuous at each $z\in S$ \cite[p.186]{Klimek}, and therefore the restriction of $V_S^*$ on $S$ is identically zero, which together with \cite[Corollary~5.1.4]{Klimek} implies that $V_S$ is continuous in $\mathbb C^d$. Hence, for $L$-regular sets $S$ we have that $g_{\Omega}=V_S=V_S^*$.

   We are now prepared to complete the proof. Since $S=\supp(\mu)$ non-pluripolar excludes $S$ to be an algebraic variety and thus $\mathcal S(\mu)=\mathbb C^d$, we may apply \eqref{def_kernel3_ter} together with the inequality $$
        \| p \|_{2,\mu}^2 \leq \mu(\mathbb C)\, \| p \|_{L^\infty(S)}^2 ,
   $$
   for $p\in \mathbb C[z]$, in order to conclude that, for all $z\in \mathbb C^d$,
   \begin{eqnarray*}
       \liminf_{m \to \infty} K_{n_{tot}(m)}^\mu(z,z)^{1/m}
       &=&\liminf_{m \to \infty} \max_{{\tdeg p \leq m}\atop{p \neq 0}} \Bigl( \frac{|p(z)|^2}{\|p\|_{2,\mu}^2} \Bigr)^{1/m}
       \\&\geq& \Phi_S(z)^2
       = \exp(2 V_S(z)),
   \end{eqnarray*}
   which shows parts (a) and (b). For our proof of part (c),  we use the additional assumption that
   $$
       \limsup_{m \to \infty} \max_{{\tdeg p \leq m}\atop{p \neq 0}} \Bigl( \frac{\|p\|_{L^\infty(S)}}{\|p\|_{2,\mu}} \Bigr)^{1/m}=1,
   $$
   implying that
   $$
       \limsup_{m \to \infty} K_{n_{tot}(m)}^\mu(z,z)^{1/m} \leq \Phi_S(z)^2
       = \exp(2 V_S(z)),
   $$
   as required for the claim of Lemma~\ref{lem_growth}(c).
\end{proof}

Since for $d=1$ the plurisubharmonic Green function becomes the classical Green function, and $n_{tot}(m)=m$, we see that Lemma~\ref{lem_growth} gives the (pointwise) multivariate equivalent of \eqref{def_kernel4}, \eqref{def_kernel4'}, at least for sufficiently smooth $S=\supp(\mu)$ and in case $S=\widehat S$, that, is, for polynomially convex supports $S$. We thus have exponential growth of the Christoffel function in $\Omega$, but not in $\supp(\mu)$. In particular, we should mention that \eqref{def_kernel4_bis'} and \eqref{def_kernel4_ter} are wrong in general for outliers $z\in \supp(\mu)$ being isolated in $\Omega$, since the plurisubharmonic Green function remains the same after adding pluripolar sets \cite[Corollary 5.2.5]{Klimek}. We will recall in \S\ref{subsection_Green} below explicit formulas for the plurisubharmonic Green function for a certain number of non-pluripolar, $L$-regular and polynomially convex sets where Lemma~\ref{lem_growth} could be applied successfully.
\color{black}

Also, we should mention that in the recent paper \cite{BosLevenberg_new} the authors discuss Siciak functions for more general degree constraints based on convex bodies, together with its pluripotential interpretation.


\subsection{Leverage scores, the Mahalanobis distance and Christoffel-Darboux kernels}\label{sec_leverage}
The aim of this section is to recall some links between the Christoffel-Darboux kernel and several notions from statistics. For a recent related work see \cite{PBV}.  In this subsection we will denote elements of $\mathbb C^d$ as row vectors. The mean $m(\mu)$ and positive definite covariance matrix $\mathcal{C}(\mu)$ of a random variable with values in $\mathbb C^d$ with law described by some probability measure $\mu$ have the entries
$$
    m(\mu)_j = \int z_j d\mu(z) , \quad
    \mathcal{C}(\mu)_{j,k} = \int \overline{(z_j-m(\mu)_j)}(z_k-m(\mu)_k) d\mu(z) .
$$
The Mahalanobis distance between a point $z\in \mathbb C^d$ and the probability measure $\mu$ is given by
$$
     \Delta(z,\mu) = \sqrt{(z-m(\mu))\mathcal{C}(\mu)^{-1}(z-m(\mu))^\ast}.
$$
The reader easily checks in terms of the moment matrices introduced in Lemma~\ref{lem_degree_OP} that, for $v_d(z)=(1,z_1,...,z_d)$, we have that
$$
     V^\ast \widetilde M_d(\mu) V = \left[\begin{array}{cc}
             \langle 1,1\rangle_{2,\mu} & 0 \\ 0 & \mathcal{C}(\mu) \end{array}\right]  , \quad V:= \left[\begin{array}{cc}
             1 & - m(\mu) \\ 0 & I \end{array}\right],
$$
where we recall that $\langle 1,1\rangle_{2,\mu}=\mu(\mathbb C^d)=1$ and hence $p_0^\mu(z)=K_0^\mu(z,z)=1$ for all $z\in \mathbb C^d$. We conclude from the previous identity that, with $\mathcal C(\mu)$, also $\widetilde M_d(\mu)=M_d(\mu)$ is invertible. Then
\begin{eqnarray} && \label{eq.Mahalanbis}
   K_d^\mu(z,z) = v_d(z) M_d(\mu)^{-1} v_d(z)^\ast
   \\&& \nonumber
   \qquad \qquad= v_d(z) V \left[\begin{array}{cc}
             1 & 0 \\ 0 & \mathcal{C}(\mu)^{-1} \end{array}\right] V^\ast v_d(z)^\ast
   = 1 + \Delta(z,\mu)^2 .
\end{eqnarray}
Hence, up to some additive constant, $\sqrt{K_n^\mu(z,z)}$ can be considered as a natural generalization of the Mahalanobis distance between a point $z\in \mathbb C^d$ and the probability measure $\mu$.

In the next statement we will introduce what will be referred to as our leverage score for each element of the support of a discrete measure.

\begin{lem}\label{lem_leverage}
   Consider the discrete probability measure
   $$
        \widetilde \nu = \sum_{j=1}^N t_j \delta_{z^{(j)}}
   $$
   for $t_j >0$ and distinct $z^{(j)}\in \mathbb C^d$. Then, for all $j=1,2,...,N$ and for all $n$ such that $p_n^{\widetilde \nu}$ exists, our {\em leverage score} $t_j K_n^{\widetilde \nu}(z^{(j)},z^{(j)})$ satisfies
   $$
         t_j K_n^{\widetilde \nu}(z^{(j)},z^{(j)}) \in (0,1] .
   $$
\end{lem}
\begin{proof}
   This is an immediate consequence of \eqref{def_kernel3_ter} and of the fact that $z^{(j)}\in \supp({\widetilde \nu})\subset \mathcal{S}({\widetilde \nu})$.
\end{proof}

For a discrete probability measure ${\widetilde \nu}$ as in Lemma~\ref{lem_leverage}, it is classical in data analysis to consider the data matrix $X \in \mathbb C^{N \times d}$ with centered and scaled rows $\sqrt{t_j} \Bigl( z^{(j)} - m({\widetilde \nu}) \Bigr)$, and thus $\mathcal{C}({\widetilde \nu})=X^\ast X$.
In the literature on data analysis, the quantity $$
     (X (X^\ast X)^{-1} X^\ast)_{j,j} = t_j \Delta(z^{(j)},{\widetilde \nu})^2
$$ for $j=1,...,N$ is known to be an element of $[0,1)$, and referred to as a leverage score: datum $z^{(j)}$ is considered to be an outlier if $(X (X^\ast X)^{-1} X^\ast)_{j,j}$ is ``close'' to $1$. Similar techniques have been applied for eliminating outlier data in multiple linear regression; see for instance Cook's distance. To make the link with our leverage score for $n=d$, recall from \eqref{eq.Mahalanbis} that
$$
     t_j \Bigl(K_d^{\widetilde \nu}(z^{(j)},z^{(j)}) - 1\Bigr) = (X (X^\ast X)^{-1} X^\ast)_{j,j} \in [0,1-t_j],
$$
the last inclusion following from $1=K_0^{\widetilde \nu}(z^{(j)},z^{(j)})\leq K_d^{\widetilde \nu}(z^{(j)},z^{(j)})$, and Lemma~\ref{lem_leverage}. As a consequence, our leverage score of Lemma~\ref{lem_leverage} for $n=d$ can be closer to $1$, and thus can be considered as slightly more precise. Choosing an index $n>d$ corresponds to considering in the regression not only random variables $X_j$ but also their powers $X^\alpha$ describing mixed effects.

It remains to show that $t_j K_n^{\widetilde \nu}(z^{(j)})$ is close to $1$ (and for large $n$ ``very'' close) iff $z^{(j)}$ is an outlier. Of course this necessitates giving a more precise meaning to what is called an outlier. These two questions are studied in the next sections; see, e.g., Corollary~\ref{cor1_Bergman} and Corollary~\ref{cor2_Bergman} for $d=1$.


\section{Approximation of Christoffel-Darboux kernels}\label{section_approximation}

The aim of the present section is to give sufficient conditions on positive measures $\mu,\nu$ with compact support in $\mathbb C^d$ insuring that the corresponding Christoffel-Darboux kernels $K_n^\mu(z,z)$ and $K_n^\nu(z,z)$ are close for a fixed $n$ and for all $z\in \mathbb C^d$. To this aim the following matrices are essential.

\begin{defi}
  For $k=k_m^\mu=k_n^\nu$, we define the matrix of mixed moments as
  $$
          R_n(\nu,\mu)
          := \int v_n^\nu(z)^\ast v_m^\mu(z) d\nu(z) \in \mathbb C^{(n+1)\times(m+1)},
  $$
  and the matrix of modified moments
  $$
          M_m(\nu,\mu)
          := \int v_m^\mu(z)^\ast v_m^\mu(z) d\nu(z) \in \mathbb C^{(m+1)\times(m+1)}.
  $$
\end{defi}

\subsection{Comparing two measures}\label{section_comparing_2_measures}

In this subsection we study conditions allowing one family of orthogonal polynomials to be linear combinations of the other.

\begin{lem}\label{lem_existence_transfer}
   Assume that two positive measures $\mu,\nu$ satisfy $k=k_m^\mu=k_n^\nu$, for some positive integers $n,m$. The following statements are equivalent:
    \\ {\rm (a)} $\mathcal N_n(\nu)\subset\mathcal N_m(\mu)$;
    \\ {\rm (b)} $\mathcal S_m(\mu)\subset\mathcal S_n(\nu)$;
    \\ {\rm (c)} $\{ k_0^\mu , k_1^\mu, k_2^\mu,...,k_m^\mu\} \subset \{ k_0^\nu , k_1^\nu,k_2^\nu,...k_n^\nu\}$;
    \\ {\rm (d)} The orthogonal polynomials $p_0^\mu,...,p_m^\mu$ can be written as a linear combination of $p_0^\nu,p_1^\nu,...,p_n^\nu$;
    \\ {\rm (e)} $v_m^\mu = v_n^\nu R_n(\nu,\mu)$;
    \\ {\rm (f)} There exists a constant $c_k$ such that, for all $p\in \mathbb C[z]$ of degree at most $k$,
    $$
          \| p \|_{2,\mu} \leq c_k \, \| p \|_{2,\nu}.
    $$
    In addition, the smallest such constant is given by $c_k = \| R_m(\mu,\nu) \|$.
\end{lem}
\begin{proof}
   We know from Lemma~\ref{lem_decomposition} that we may write the set of multivariate polynomials of degree at most $k=k_m^\mu=k_n^\nu$ as a direct sum $\mathcal L_m(\mu)\oplus \mathcal N_m(\mu)=\mathcal L_n(\nu)\oplus \mathcal N_n(\nu)$. Thus the equivalence between (a), (b), (c) and (d) follows immediately from Definition~\ref{def_sets}.
   The implication (d) $\Longrightarrow$ (e) follows by $\nu$-orthogonality, and the implication (f) $\Longrightarrow$ (a) is trivial. Hence it only remains to show that (e) implies (f).

   Suppose that (e) holds, implying $\mathcal L_m(\mu)\subset \mathcal L_n(\nu)$. Let $p\in \mathbb C[z]$ of degree at most $k$. Applying  Lemma~\ref{lem_decomposition} we may write
   $$
        p = q + p_3 , \quad
        q = v_n^\nu \xi \in \mathcal L_n(\nu), \quad
        p_3 \in \mathcal N_n(\nu) \subset \mathcal N_m(\mu)
   $$
   for some $\xi\in \mathbb C^{n+1}$.
   Then, applying again Lemma~\ref{lem_decomposition},
   $$
        q = p_1+p_2,  \quad
        p_1 \in \mathcal L_m(\mu), \quad
        p_2 \in \mathcal N_m(\mu) \cap \mathcal L_n(\nu),
   $$
   where
   \begin{eqnarray*}
       p_1(z) &=& \int K_n^\mu(z,w) q(w) d\mu(w)
       \\&=& \int v_n^\mu(z) v_n^\mu(w)^\ast v_n^\nu(w)\xi d\mu(w) = v_n^\mu(z) R_m(\mu,\nu) \xi .
   \end{eqnarray*}
   Hence
   \begin{eqnarray*}
        \| p \|_{2,\mu}
        &=& \| p_1 \|_{2,\mu} = \| R_m(\mu,\nu) \xi \|
        \leq \| R_m(\mu,\nu) \| \, \| \xi \|
        \\&=&
        \| R_m(\mu,\nu) \| \, \| q \|_{2,\nu}  = \| R_m(\mu,\nu) \| \, \| p \|_{2,\nu},
   \end{eqnarray*}
   showing that (e) holds with $c_k=\| R_m(\mu,\nu) \|$. Taking the supremum for $\xi \in \mathbb C^{n+1}\setminus \{ 0 \}$ shows that $c_k$ cannot be smaller.
\end{proof}

Under the assumptions of Lemma~\ref{lem_existence_transfer}(e), we can make a link between $R_n(\nu,\mu)$ and the matrices defined in Lemma~\ref{lem_degree_OP}, namely
\begin{equation} \label{eq_mixed_moments}
      R_n(\nu,\mu) = R_n(\nu)^{-1} E_k(\nu)^T E_k(\mu) R_m(\mu).
\end{equation}
To see this, recall from \eqref{eq.degree_OP3} that
$$
       v_m^\mu = v_k E_k(\mu) R_m(\mu),
       \quad
       v_n^\nu = v_k E_k(\nu) R_n(\nu),
$$
where Lemma~\ref{lem_existence_transfer}(c) tells us that $E_k(\mu)=E_k(\nu)E_k(\nu)^T E_k(\mu)$, and thus
$$
      v_m^\mu = v_n^\nu R_n(\nu)^{-1} E_k(\nu)^T E_k(\mu) R_m(\mu)
      = v_n^\nu R_n(\nu,\mu) ,
$$
implying \eqref{eq_mixed_moments}.
In this case, it is not difficult to check that $R_n(\nu,\mu)$ is of full column rank $m+1$ and of upper echelon form, and that we have full rank decomposition
\begin{equation} \label{eq_cholesky}
       M_m(\nu,\mu)=R_n(\nu,\mu)^*R_n(\nu,\mu).
\end{equation}

In the remainder of this section we will always suppose that $k=k_m^\mu=k_n^\nu$ with $\mathcal N_n(\nu)=\mathcal N_m(\mu)$. Thus we may apply Lemma~\ref{lem_existence_transfer} twice, implying that $m=n$ and $k_j^\mu=k_j^\nu$ for $j=0,1,...,n$ by Lemma~\ref{lem_existence_transfer}(c). In particular, we find that $E_k(\nu)^T E_k(\mu)$ is the identity of order $n+1$. We conclude using \eqref{eq_mixed_moments} that $R_n(\nu,\mu)$ is upper triangular and invertible, with inverse given by
\begin{equation} \label{eq_transfer_invertible}
      R_n(\nu,\mu)^{-1} = R_n(\mu)^{-1} R_n(\nu) = R_n(\mu,\nu) .
\end{equation}
We summarize these findings and their consequence for the Christoffel-Darboux kernel in the following statement.

\begin{cor}\label{cor_Gram}
   Let $n\geq 0$ be some integer\footnote{The interested reader might want to check that this condition is true for all $n$ if $\mathcal N(\mu)=\mathcal N(\nu)$.} with $k_n^\mu=k_n^\nu$ and $\mathcal N_n(\mu)=\mathcal N_n(\nu)$. Then
   \begin{equation} \label{eq.transition}
           v_n^\mu(z)=v_n^\nu(z) R_n(\nu,\mu)
           , \quad
           v_n^\nu(z)=v_n^\mu(z) R_n(\mu,\nu)
           ,
   \end{equation}
   with the upper triangular and invertible matrices $R_n(\nu,\mu)=R_n(\nu)^{-1} R_n(\mu)$ and $R_n(\mu,\nu)=R_n(\nu,\mu)^{-1}$, allowing for the Cholesky decompositions
   \begin{equation} \label{eq.transition.gramm}
        M_n(\nu,\mu) = R_n(\nu,\mu)^\ast R_n(\nu,\mu),
        \quad
        M_n(\mu,\nu) = R_n(\mu,\nu)^\ast R_n(\mu,\nu).
   \end{equation}
   In particular,
   \begin{equation} \label{eq.schur_kernel}
        K_n^\nu(z,w) = v_n^\mu(z) M_n(\nu,\mu)^{-1} v_n^\mu(w)^\ast .
   \end{equation}
\end{cor}


\begin{examples}\label{ex_transfer}
   (a) Let $\mu=\mu_0$ be the normalized Lebesgue measure on $(\partial \mathbb D)^d$ making monomials orthonormal. Here $\mathcal N(\mu_0)=\{ 0 \}$, and thus $k_n^{\mu_0}=n$ for all $n$. As a consequence, Lemma~\ref{lem_existence_transfer}(d),(e) reduces to Lemma~\ref{lem_degree_OP}.

   (b) Consider the case $\nu=\mu+\sigma$ for some positive measure $\sigma$ with compact support; for instance, the case of adding a finite number of point masses.
   Since $\nu-\mu\geq 0$, clearly $\mathcal N(\nu)\subset \mathcal N(\mu)$, and from Lemma~\ref{lem_existence_transfer} we know that we may express the orthogonal polynomials $p_j^\mu$ in terms of the $p_j^\nu$. Moreover, $\| M_n(\mu,\nu) \| =\| R_m(\mu,\nu) \|^2 \leq 1$ by Lemma~\ref{lem_existence_transfer}(f).
   If,
     in addition, 
   $\supp(\sigma)\subset \mathcal S(\mu)$, then $\mathcal N(\mu)=\mathcal N(\mu+\sigma)$ by Example~\ref{exa_supp}. Thus, in this case, we may also express the orthogonal polynomials $p_j^\nu$ in terms of the $p_j^\mu$ via  \eqref{eq.transition}, $k_j^\mu=k_j^\nu$ for all $j\geq 0$, and, by \eqref{eq.transition.gramm} and \eqref{eq.schur_kernel},
   $$
       \forall z\in \mathbb C^d, \, \forall n \geq 0 : \quad
           K_n^{\mu+\sigma}(z,z) \leq K_n^\mu(z,z).
   $$
\end{examples}

\begin{rems}\label{rem_outliers1}
   Lasserre and Pauwels in {\cite[Theorem~3.9 and Theorem~3.12]{LP2}} consider the graded lexicographical ordering \eqref{total_degree}
   and
   give an explicit sequence of numbers $\kappa_m$ such that the Hausdorff
    distance $d_H(S,S_m)$ between the sets $$
      S=\supp(\mu), \mbox{~~~~~and~~~~~} S_m=\{ x \in \mathbb R^n : K_{n_{tot}(m)}^\mu(z,z) \leq \kappa_m \}
   $$ tends to $0$. For this they assume that $S=\mbox{Clos}(\mbox{Int}(S)))$, and that there is some constant $w_{\min}>0$ such that $d\mu \geq w_{\min} dA|_S$, where $dA$ denotes (unnormalized) Lebesgue measure in $\mathbb R^d$.

  We summarize their reasoning, by using our notation. All measures $\nu$ considered in this context are supported in $\mathbb R^d$, with $\supp(\nu)$ having non-empty interior (in $\mathbb R^d$), such that $\mathcal N(\nu)=\{ 0 \}$ and $\mathcal S(\nu)=\mathbb C^d$ by Example~\ref{exa_supp}.
   Notice that $d_H(S,S_m)\leq \delta$ if
   \begin{eqnarray} &&\label{eqLP1}
       \forall z \in \supp(\mu) \mbox{~with $\mbox{dist}(z,\partial \supp(\mu))\geq \delta$:} \quad  K_{n_{tot}(m)}^\mu(z,z) \leq \kappa_m,
       \\&&\label{eqLP2}
       \forall z \not\in \supp(\mu) \mbox{~with $\mbox{dist}(z,\partial \supp(\mu)) \geq \delta$:} \quad  K_{n_{tot}(m)}^\mu(z,z) \geq \kappa_m.
   \end{eqnarray}
   For $z$ as in \eqref{eqLP1} we denote by $B$ the real unit ball in $\mathbb R^n$, and observe that $z+\delta B\subset S$. Hence
   $
          w_{\min} \, A|_{z+\delta B} \leq
          w_{\min} \, A|_{S} \leq \mu.
   $
   Using the fact that that an explicit upper bound $\kappa_m'$ is known for $K_{n_{tot}(m)}^{A|_B}(0,0)$ since the 90-ies, we obtain
   by Example~\ref{ex_transfer}(b) and \eqref{affine_invariance}
   $$
        K^\mu_{n_{tot}(m)}(z,z) \leq \frac{K^{A|_{z+\delta B}}_{n_{tot}(m)}(z,z)}{w_{\min}} = \frac{A(B)}{A(z+\delta B)} \frac{K^{A|_{B}}_{n_{tot}(m)}(0,0)}{w_{\min}} = \frac{\kappa'_m}{\delta^d \, w_{\min}},
   $$
   compare with \cite[Lemma~6.1]{LP2}. With $z$ as in \eqref{eqLP2}, they consider the annulus $S'=\{ x\in \mathbb R^d: \delta \leq \| x \| \leq r\}$, where $r=\delta+\mbox{diam}(S)$ such that $S \subset S'$. Following \cite{Kroo-Lubinsky-1,Kroo-Lubinsky-2}, they then consider the real {\it needle polynomial}
   $$
      p(x)= T_{\lfloor m/2\rfloor}(\frac{r^2+\delta^2 - 2 \| x - z \|^2}{r^2-\delta^2}), \quad \mbox{with} \quad
      \| p \|_{L^\infty(S)} \leq \| p \|_{L^\infty(S')} = 1,
   $$
   $\tdeg p \leq m$ and thus $\deg p \leq n_{tot}(m)$, and
   $$
      |p(z)|=T_{\lfloor m/2\rfloor}(\frac{r^2+\delta^2}{r^2-\delta^2})
      \geq \frac{1}{2} e^{\lfloor m/2\rfloor g_{\mathbb C\setminus [-1,1]}(\frac{r^2+\delta^2}{r^2-\delta^2})}
      \geq \frac{1}{2} \Bigl( \frac{r+\delta}{r-\delta}\Bigr)^{(m-1)/2};
   $$
   compare with \cite[Lemma~6.3]{LP2}. It follows from \eqref{def_kernel3_ter} that
   $$
        K_{n_{tot}(m)}^\mu(z,z) \geq \frac{|p(z)|^2}{\| p \|_{2,\mu}^2}
        \geq \frac{1}{2 \mu(\mathbb C)} \Bigl( 1 + \frac{2\delta}{\mbox{diam}(S)}\Bigr)^{m-1} .
   $$
   In view of \eqref{eqLP1}, \eqref{eqLP2}, it thus only remains to find $\delta$ as a function of $m$ (possibly decreasing) and $\kappa_m$ such that
   $$
          \frac{\kappa'_m}{\delta_m^d \, w_{\min}} \leq \kappa_m \leq \frac{1}{2 \mu(\mathbb C)} \Bigl( 1 + \frac{2\delta_m}{\mbox{diam}(S)}\Bigr)^{m-1} .
   $$
   For instance, for $\delta_m=1/m^\alpha$ with $\alpha\in (0,1)$, the left-hand term grows like a power of $m$, and the right-hand side exponentially in $m$, so that the above inequalities are true for sufficiently large $m$.
   Enclosed we list some possible improvements.
   \begin{enumerate}
   \item The assumption $S=\mbox{Clos}(\mbox{Int}(S)))$ for $S=\supp(\mu)$ is strongly used, which does not allow outliers.
   \item The assumption $S=\mbox{Clos}(\mbox{Int}(S)))$ is strongly used also to give upper bounds for $K_{n_{tot}(m)}^\mu(z,z)$ in $\supp(\mu)$ at some distance from the boundary. This condition could be replaced by stepping to
       the relative interior of $\supp(\mu)$ in $\mathcal S(\mu)\cap \mathbb R^d$ (in the real setting) or $\mathcal S(\mu)$ in the complex setting. This may possibly result in an increase in the power of $m$.
   \item The lower bound for $K_{n_{tot}(m)}^\mu(z,z)$ is very rough, and it should be possible to write something down in terms of the plurisubharmonic Green function of $\supp(\mu)$. One difficulty arises in this direction: one needs bounds uniformly in $z$ for all $z\in \supp(\mu)$ with distance to $\supp(\mu)$ decaying like a fractional negative power of $m$.
   \end{enumerate}
\end{rems}

\begin{rems}
   Modified moment matrices and mixed moment matrices have been shown to be useful for explicitly computing orthogonal polynomials of one or several complex variables; see for instance the modified Chebyshev algorithm for orthogonality on the real line \cite{Gautschi}.
   Here the basic idea is that one knows $\mu$ and its orthogonal polynomials $p_j^\mu$, and computes both $R_n(\mu,\nu)$ and $v_n^\nu(z)=v_n^\mu(z) R_n(\mu,\nu)$ through $\nu$-orthogonality, using, e.g., the Gram-Schmidt method. In case of finite precision arithmetic, we can only insure small errors if $R_n(\mu,\nu)$ has a modest condition number given by $\| R_n(\mu,\nu) \| \, \| R_n(\mu,\nu)^{-1} \|$. In case of one complex variable it has been shown in \cite[Lemma~3.4]{Beckermann_Habil} that this condition number grows exponentially in $n$ unless the complements of $\supp(\mu)$ and $\supp(\nu)$ have the same unbounded connected component, and, according to Lemma~\ref{lem_growth}, a similar result is expected to be true in the case of several complex variables. Thus the choice of $\mu$ is essential.

   In particular, it is not always a good idea to compute $p_0^\nu,p_1^\nu,...$ from monomials, even after scaling \eqref{affine_invariance}; compare with Example~\ref{ex_transfer}(a). But the modified moment matrix $M_n(\mu,\nu)$ and hence $R_n(\mu,\nu)$ might be much better conditioned for $\mu$ close to $\nu$, or, in other words, the $p_j^\mu$ are ``nearly" $\nu$-orthogonal for $j=0,1,...,n$.
\end{rems}


\subsection{Small perturbations}\label{section_small_perturbations}

In this subsection we compare for fixed $n$ the Christoffel-Darboux kernels
for two measures $\mu$ and $\nu$ that are close in a precise sense.
In view of Lemma~\ref{lem_existence_transfer}(f), the proximity of the two measures is imposed by the following condition: there exists a sufficiently small $\epsilon\in (0,1)$ such that
\begin{equation} \label{eq.close_measures}
      (1-\epsilon) \, \| p \|^2_{2,\mu}\leq \| p \|^2_{2,\nu} \leq (1+\epsilon) \, \| p \|^2_{2,\mu}
\end{equation}
for all polynomials $p$ of degree at most $k_n^\mu$. As a consequence of \eqref{eq.close_measures}, the quantity $\| p \|^2_{2,\mu}$ vanishes if and only if this is true for  $\| p \|^2_{2,\nu}$; that is,
$k_n^\mu=k_n^\nu$ and $\mathcal N_{n}(\mu)=\mathcal N_{n}(\nu)$.
This allows us to apply Corollary~\ref{cor_Gram}, in particular $v_n^\mu=(p_0^\mu,...,p_n^\mu)=v_n^\nu R_n(\nu,\mu) = (p_0^\nu,...,p_n^\nu) R_n(\nu,\mu)$ with the Hermitian and positive definite modified moment matrix $M_n(\nu,\mu)=R_n(\nu,\mu)^\ast R_n(\nu,\mu)$ being similar to $M_n(\mu,\nu)^{-1}$.

This modified moment matrix allows us to restate assumption \eqref{eq.close_measures}: for any $\xi \in \mathbb C^{n+1}$, by considering $p(z)=v_n^\mu(z) \xi=v_n^\nu(z) R_n(\nu,\mu) \xi$ in \eqref{eq.close_measures}, we find
\begin{equation} \label{eq.Rayleigh}
      (1-\epsilon) \|\xi\|^2 \leq
      \| R_n(\nu,\mu)\xi \|^2 = \xi^\ast M_n(\nu,\mu)\xi \leq (1+\epsilon) \|\xi \|^2;
\end{equation}
in other words, $\nu$ is so close to $\mu$ that all eigenvalues of the Hermitian matrix $M_n(\nu,\mu)-I$ lie in the interval $[-\epsilon,\epsilon]$, or $\| M_n(\nu,\mu)-I \| \leq \epsilon$. Using the Froebenius matrix norm, we therefore get the sufficient condition
\begin{equation} \label{eq.Rayleigh_bis}
     \sum_{j,k=0}^n \Bigl| \langle p_j^\mu,p_k^\mu\rangle_{2,\nu} - \delta_{j,k} \Bigr|^2 = \| M_n(\nu,\mu)-I \|_F^2 \leq \epsilon^2
\end{equation}
which indicates how $p_0^\mu,...,p_n^\mu$ are ``nearly" $\nu$-orthogonal.\

For two Hermitian matrices $X_1,X_2$ we write that $X_1 \leq X_2$ if $X_2-X_1$ is positive semi-definite. We make use of this notation in the proofs of the following results.

\begin{prop}\label{prop.2measures}
   Under the assumption \eqref{eq.close_measures} there holds for all $z,w\in \mathbb C^d$
   \begin{eqnarray} && \label{prop.2measures1}
       (1-\epsilon) \, K^\nu_n(z,z)
       \leq K^\mu_n(z,z) \leq
       (1+\epsilon) K^\nu_n(z,z) ,
       \\&& \label{prop.2measures2}
       |K^\mu_n(z,w) - K^\nu_n(z,w)| \leq \epsilon \sqrt{K^\nu_n(z,z)} \sqrt{K^\nu_n(w,w)},
       \\&& \label{prop.2measures3}
       |C_n^\mu(z,w) - C_n^\nu(z,w) |
         \leq 2 \epsilon.
   \end{eqnarray}
\end{prop}
\begin{proof}
   Recall from \eqref{eq.schur_kernel} that
   \begin{eqnarray} \nonumber
     K_n^\mu(z,w) =
     v_n^\nu(z) M_n(\mu,\nu)^{-1} v_n^\nu(w)^\ast .
   \end{eqnarray}
   Since $M_n(\mu,\nu)^{-1}$ is similar to $M_n(\nu,\mu)$, we get from \eqref{eq.Rayleigh} that
   \begin{equation} \label{eq.2kernel2}
        (1-\epsilon) I \leq M_n(\mu,\nu)^{-1} \leq (1+\epsilon) I ,
   \end{equation}
   and a combination with \eqref{eq.schur_kernel} for $w=z$ gives \eqref{prop.2measures1}.
   Moreover, it follows from \eqref{eq.2kernel2} that $\| M_n(\mu,\nu)^{-1} - I \| \leq \epsilon$, and thus, again by \eqref{eq.schur_kernel},
   $$
        |K^\mu_n(z,w) - K^\nu_n(z,w)| \leq \epsilon \| v_n^\nu(z) \| \, \| v_n^\nu(w) \|,
   $$
   implying \eqref{prop.2measures2}.
   Finally, from \eqref{prop.2measures1} we infer
   $$
        \frac{K_n^\mu(z,z)K_n^\mu(w,w)}{K_n^\nu(z,z)K_n^\nu(w,w)} \in \Bigl[(1-\epsilon)^2,(1+\epsilon)^2\Bigr]
   $$
   and thus, by definition \eqref{cos} and \eqref{prop.2measures2},
   \begin{eqnarray*} &&
       |C_n^\mu(z,w) - C_n^\nu(z,w) |
       \\&&\leq \left| \frac{K_n^\mu(z,w)-K_n^\nu(z,w)}{\sqrt{K_n^\nu(z,z)K_n^\nu(w,w)}}\right|
       + |C_n^\mu(z,w) | \, \left|
       1 - \sqrt{\frac{K_n^\mu(z,z)K_n^\mu(w,w)}{K_n^\nu(z,z)K_n^\nu(w,w)}}
       \right|
       \\&& \leq \epsilon + |C_n^\mu(z,w) | \,\epsilon \leq 2 \epsilon,
   \end{eqnarray*}
   as claimed in \eqref{prop.2measures3}.
\end{proof}

\begin{rem}\label{eq.Rayleigh_mass}
    If assumption \eqref{eq.close_measures} holds for a pair of measures $(\mu,\nu)$, then it also holds for $(\mu+\sigma,\nu+\sigma)$ for any measure $\sigma\geq 0$ supported on a subset of  $\mathcal{S}(\mu)$, and hence
    $$
       (1-\epsilon) \, K^{\nu+\sigma}_n(z,z)
       \leq K^{\mu+\sigma}_n(z,z) \leq
       (1+\epsilon) K^{\nu+\sigma}_n(z,z)
    $$
    by Proposition~\ref{prop.2measures}. Indeed, from Example~\ref{ex_transfer}(b) we know that $k_n^\mu=k_n^{\mu+\sigma}$, and so, for any polynomial $p,$
    of degree at most $k_n^{\mu+\sigma}$
    $$
         | \| p \|^2_{2,\mu+\sigma} - \| p \|^2_{2,\nu+\sigma} | =
         | \| p \|^2_{2,\mu} - \| p \|^2_{2,\nu} |  \leq \epsilon \, \| p \|^2_{2,\mu}\leq \epsilon \, \| p \|^2_{2,\mu+\sigma},
    $$
    which implies assumption \eqref{eq.close_measures} for $(\mu+\sigma,\nu+\sigma)$.
\end{rem}


\begin{rem} \label{remark3.9}
     Multi-point analogues of Proposition~\ref{prop.2measures} are also available. We include them for completeness.
    Suppose that assumption \eqref{eq.close_measures} holds, and let $z_1,...,z_\ell\in \mathbb C^d$.
    Then
    \begin{eqnarray*}
              (1-\epsilon) \, K^\nu_n(z_1,...,z_\ell;z_1,...,z_\ell)
       &\leq&
       K^\mu_n(z_1,...,z_\ell;z_1,...,z_\ell)
       \\&\leq&
       (1+\epsilon) K^\nu_n(z_1,...,z_\ell;z_1,...,z_\ell) ,
    \end{eqnarray*}
    and, for the Froebenius matrix norm $\| \cdot \|_F,$
    \begin{eqnarray*}
              \| C^\mu_n(z_1,...,z_\ell;z_1,...,z_\ell) - C^\nu_n(z_1,...,z_\ell;z_1,...,z_\ell) \|_F
              \leq 2 \sqrt{\ell(\ell-1)} \epsilon.
    \end{eqnarray*}
    For a proof, consider the matrix
    $$
        V^\nu_n(z_1,...,z_\ell) := \left[\begin{array}{cc}
        v^\nu_n(z_1) \\ \vdots \\ v^\nu_n(z_\ell) \end{array}\right] .
    $$
    From the factorization
    $$
       K^\nu_n(z_1,...,z_\ell;z_1,...,z_\ell)=
       V^\nu_n(z_1,...,z_\ell) V^\nu_n(z_1,...,z_\ell)^\ast,
    $$
    the identity for the polarized kernel follows by multiplying
    \eqref{eq.2kernel2} on the left by $\xi^* V^\nu_n(z_1,...,z_\ell)$, and on the right by its adjoint. For the cosine identity, we apply \eqref{prop.2measures3} and obtain
    \begin{eqnarray*} &&
       \| C^\mu_n(z_1,...,z_\ell;z_1,...,z_\ell) - C^\nu_n(z_1,...,z_\ell;z_1,...,z_\ell) \|^2_F
       \\&& = \sum_{j,k=1,j\neq k}^\ell | C^\mu_n(z_j,z_k) - C^\nu_n(z_j,z_k) |^2
         \leq 4 \ell(\ell-1) \epsilon^2 ,
    \end{eqnarray*}
    as required for the above claim.
\end{rem}

\begin{rems}
   In a series of papers, Migliorati and his co-authors considered a fixed general probability measure $\mu$ compactly supported in $\mathbb R^d$ (probably their results still remain valid in $\mathbb C^d$) and the discrete measure
   $$
         \nu = \frac{1}{N} \sum_{j=1}^N w(z^{(j)}) \delta_{z^{(j)}} ,
   $$
   with the $N$ masspoints $z^{(j)}\in \supp(\mu)$ being independent and identically distributed random variables with law given by some sampling probability measure $\widetilde \mu$ with $\supp(\widetilde \mu)=\supp(\mu)$ (e.g., $\mu=\widetilde \mu$), and the density
   $w(z)=\frac{d\mu}{d\widetilde \mu}(z)$ is assumed to be strictly positive in $\supp(\mu)$.    We now refer to \cite{cohen} where, as in the present paper, the authors allow for general degree sequences and general measures $\mu$.  In particular, in  \cite[Theorem~2.1(i)]{cohen}, the authors choose $N$ large enough such that
   $$
       \max_{z\in \supp(\mu)} w(z) K_n^\mu(z,z) \leq \frac{1-\log2}{2+2r} \frac{N}{\log(N)}
   $$
   for some $r>0$, and show that the probability that $\| M_n(\nu,\mu)-I\| >1/2$ is less than $2 N^{-r}$. In other words, condition \eqref{eq.Rayleigh} holds for $\epsilon = 1/2$ with a probability $1-2N^{-r}$ close to $1$. By checking the proof, it can be seen that a similar result is true for any $\epsilon\in (0,1/2)$.
\end{rems}

\section{Additive perturbation of positive measures}\label{section_additive}

The aim of this section is to give upper and lower bounds for the ratio of Christoffel-Darboux kernels $K_n^{\mu+\sigma}(z,z)/K_{n}^\mu(z,z)$ for $z\not\in \supp(\mu)$, but possibly $z\in \supp(\sigma)$. Several  of the examples studied in later sections require that the cosine function $C_n^\mu(z,w)$ (see \eqref{cos})
 has a limit, at least along subsequences, for $z,w \not\in \supp(\mu)$ . Hence our bounds will be formulated in terms of this cosine.


Let $\supp(\mu)$ be an infinite set such that there exists an infinite number of orthogonal polynomials $p_j^\mu$, and consider the case of adding $\ell$ disjoint point masses in the Zariski closure
$\mathcal S(\mu)$ of the support of the original measure $\mu$ (and later outside $\supp(\mu)$). That is,
\begin{equation}\label{sig}
   \sigma=\sum_{j=1}^\ell t_j \delta_{z_j} , \quad t_j>0 , \quad
   z_1,...,z_\ell \in \mathcal S(\mu) \mbox{~~disjoint.}
\end{equation}
As in Remark \ref{remark3.9}, we depart from the canonical notation and allow $z_1,\ldots,z_\ell$ to be elements in $\C^d$ rather than the numerical coordinates of a single point $z$.
From Example~\ref{ex_transfer}(b) we know that $\mathcal N(\mu)=\mathcal N(\mu+\sigma)$, and thus we may apply Corollary~\ref{cor_Gram} and in particular property \eqref{eq.schur_kernel} for $\nu=\mu+\sigma$.

\begin{lem}\label{lem_K_invertible}
   If  $z_1,...,z_\ell \in \mathcal S(\mu)$ are distinct, then
   there exists an $N$ such that
   the matrix $K_n^\mu(z_1,...,z_\ell;z_1,...,z_\ell)$ is invertible for all $n \geq N$.
\end{lem}
\begin{proof}
   There exist multivariate Lagrange polynomials $p_1,...,p_\ell$, each of them of minimal degree, such that
   $p_j(z_k)=\delta_{j,k}$ for $k=1,...,\ell$. The relation $z_1,...,z_\ell \in \mathcal S(\mu)$ implies that for each $p\in \mathcal N(\mu)$ there holds $p(z_j)=0$ for $j=1,2,...,\ell$. For each $p_j$ there exists $q_j\in \mbox{span} \{ p_0^\mu,p_1^\mu,...\}$ with $\deg q_j\leq \deg p_j$ and $p_j-q_j\in \mathcal N(\mu)$. Then $q_j(z_k)= p_j(z_k)=\delta_{j,k}$, and thus also $q_1,...,q_\ell$ are Lagrange polynomials. In particular, $\deg q_j=\deg p_j$ by minimality of $\deg p_j$. We conclude that there exists an integer $N$ with $k_N^\mu=\max \{ \deg q_j^\mu : j=1,...,\ell \}$, and
   the rows $\{ v_n^\mu(z_j):j=1,2,...,\ell \}$ are linearly independent for $n\geq N$. This is false for $n<N$ by minimality of the degrees. Thus
   $K_n^\mu(z_1,...,z_\ell;z_1,...,z_\ell)$ is invertible for all $n \geq N$, but not for $n<N$.
\end{proof}

Notice that Lemma~\ref{lem_K_invertible} is trivial for the case $d=1$ of univariate polynomials, where $z^{\alpha(j)}=z^j$ and hence $N=\ell-1$.

Recalling Definition~\ref{def_multi_kernel} we see that the matrices $K_n^\mu(z_1,...,z_\ell;z_1,...,z_\ell)$ and $C_n^\mu(z_1,...,z_\ell;z_1,...,z_\ell)$ are invertible for distinct points $z_j$. This enables us to establish the following comparison
theorem consisting of a sequence of alternating inequalities.


\begin{thm}\label{thm_alternation}
   Let $z_1,...,z_\ell \in \mathcal S(\mu)$ be distinct, $n \geq N$ as in Lemma~\ref{lem_K_invertible},  and consider the following matrices (depending on $n$)
   \begin{eqnarray*} &&
      C := C_n^\mu(z_1,...,z_\ell;z_1,...,z_\ell),
      \\ &&
      \widetilde C := C_n^\mu(z_1,...,z_\ell,z;z_1,...,z_\ell,z)
      = \left[\begin{array}{cc} C & b^* \\ b & 1
         \end{array}\right], \quad b \in \mathbb C^{1 \times \ell},
      \\&&   D: = \diag\left( \frac{1}{\sqrt{t_j K_n^\mu(z_j,z_j)}}\right)_{j=1,...,\ell},\\
      and \,\,the\,\, constants
      \\&& \Sigma_m := 1 - \sum_{j=0}^{m-1} (-1)^{j} b C^{-1} (D^2C^{-1})^j b^*,\,\, m=1,2,\ldots.
   \end{eqnarray*}
   Then, for all $z\in \mathbb C^d$,
   \begin{equation} \label{thm_alternation3}
         \frac{K_n^{\mu+\sigma}(z,z)}{K_n^\mu(z,z)}
         = 1 - b (D^2 + C )^{-1} b^* ,
   \end{equation}
   and
   \begin{equation} \label{thm_alternation1}
      \Sigma_1 \leq \Sigma_3 \leq\Sigma_5 \leq ... \leq
      \frac{K_n^{\mu+\sigma}(z,z)}{K_n^\mu(z,z)}
    \leq ... \leq \Sigma_4 \leq \Sigma_2 \leq \Sigma_0.
   \end{equation}
\end{thm}
Before presenting the proof of this theorem, we make some important observations and provide two
examples where the result is applied.\

Notice that, in the case of outlier detection, $z_j$ lies in $\Omega\cap \mathcal S(\mu)$, with $\Omega$ the unbounded connected component of $\mathbb C \setminus \supp(\mu)$. Hence \eqref{def_kernel4} in case $d=1$ and \eqref{def_kernel4_bis} in case $d>1$ tell us that $\| D \|\to 0$ exponentially fast for $n\to \infty$. Also, $\| b \| \leq  \sqrt{\ell}$ since $|C_n^\mu(z,z_j)|\leq 1$. Hence, as long as $C^{-1}$ is bounded uniformly for $n$ sufficiently large, we see that \eqref{thm_alternation1} provides quite sharp lower and upper bounds, even for modest values of $m$. This assumption on $C$ is verified in many special cases discussed below where we show that (a unitary scaled counterpart of) $C$ has a finite limit as $n\to \infty$.

In the next two examples, we make this last assertion a bit more explicit.

\begin{example}
  We begin with the simple special case $\sigma=t_1\delta_{z_1}$ of adding one point mass.
  Here, with the notation of Theorem~\ref{thm_alternation},
  $$
   C=1,\ \ b = C^\mu_n(z,z_1), \ \ D^2 = \frac{1}{t_1 K^\mu_n(z_1,z_1)} .
  $$
  With these data we find from \eqref{thm_alternation3}
  $$
     \frac{K^{\mu+\sigma}_n(z,z)}{K^\mu_n(z,z)} = 1 - b(1+ D^2)^{-1} b^\ast = 1- \frac{|C^\mu_n(z,z_1)|^2}{1+ \frac{1}{t_1 K^\mu_n(z_1,z_1)}},
  $$
  and, from \eqref{thm_alternation1} for $m=2,$
  \begin{eqnarray*} &&
       - \frac{1}{(t_1 K_n^{\mu}(z_1,z_1))^2}  \leq  \Sigma_3 - \Sigma_2 \leq
       \frac{K^{\mu+\sigma}_n(z,z)}{K^\mu_n(z,z)} - \Sigma_2
              \\
&=&
       \frac{K^{\mu+\sigma}_n(z,z)}{K^\mu_n(z,z)} - 1 + |C_n^\mu(z,z_1)|^2 \Bigl(1-\frac{1}{t_1 K_n^{\mu}(z_1,z_1)}\Bigr) \leq 0.
  \end{eqnarray*}
  \qed
\end{example}

\begin{example} For $\sigma$ consisting of $\ell \geq 2$ point masses
as in \eqref{sig},  Theorem~\ref{thm_alternation} provides the following useful
estimates when taking
$m\in \{ 0,1,2\}.$
  From \eqref{thm_alternation1} for $m=0,1$ we observe that
  \begin{equation} \label{thm_alternation11}
   \Sigma_1 = \frac{\det C_n^\mu(z_1,...,z_\ell,z;z_1,...,z_\ell,z)}{\det C_n^\mu(z_1,...,z_\ell;z_1,...,z_\ell)}
   \leq \frac{K_n^{\mu+\sigma}(z,z)}{K_n^\mu(z,z)} \leq \Sigma_0 = 1 .
  \end{equation}
  where, in the equality on the left, we have used Schur complement techniques.
  Since $\Sigma_1$ vanishes when $z$ is  one of the point masses at $z_j$, we go one term further:
  \begin{equation} \label{thm_alternation12}
         - \Bigl( \Sigma_2 - \Sigma_1 \Bigr) \leq \frac{K_n^{\mu+\sigma}(z,z)}{K_n^\mu(z,z)} - \Sigma_2 \leq 0 ,
  \end{equation}
  with
  \begin{eqnarray} && \label{thm_alternation13}
    \Sigma_2 - \Sigma_1 = b C^{-1} D^2 C^{-1} b^*
    = \sum_{j=1}^\ell \frac{| b C^{-1} e_j |^2}{t_j K_n^\mu(z_j,z_j)}
    \\&& \nonumber
    = \sum_{j=1}^\ell \frac{|\det C_n^\mu(z_1,..,z_{j-1} z,z_{j+1},...,z_\ell;z_1,...,z_\ell)|^2}{|\det C_n^\mu(z_1,...,z_\ell;z_1,...,z_\ell)|^2}\frac{1}{t_j K_n^\mu(z_j,z_j)},
  \end{eqnarray}
  where in the last equality we have applied Cramer's rule.
  \qed
\end{example}

\begin{proof}[Proof of Theorem~\ref{thm_alternation}]
   We start by recalling that, by \eqref{eq.schur_kernel} for $\nu=\mu+\sigma$,
   \begin{eqnarray*} &&
      K_n^{\mu+\sigma}(z,z)
      = v_n^{\mu}(z) M_n(\mu+\sigma,\mu)^{-1} v_n^{\mu}(z)^* .
   \end{eqnarray*}
   By introducing the matrices $$
       V:=\left[\begin{array}{cc}
         v_n^\mu(z_1)/\sqrt{K^{\mu}_n(z_1,z_1)} \\
         \vdots \\ v_n^\mu(z_\ell)/\sqrt{K_n^\mu(z_\ell,z_\ell)}
      \end{array}\right] \in \mathbb C^{\ell \times (n+1)}
      , \quad
   $$
   and $D,C,b$ as in the statement of Theorem~\ref{thm_alternation},
   it is not difficult to check that
   $$
        M_n(\mu+\sigma,\mu) = I + V^* D^{-2} V ,
   $$
   and thus, by the Sherman-Morrison formula \cite[\S 2.1.4]{golub},
   $$
        M_n(\mu+\sigma,\mu)^{-1} = I - V^* D^{-1} ( I + D^{-1} V V^* D^{-1} )^{-1} D^{-1} V .
   $$
   Observing that
   \begin{eqnarray*} &&
        C=V V^* , \quad
        b=\frac{v_n^{\mu}(z) V_n^*}{\sqrt{K_n^{\mu}(z,z)}} ,
   \end{eqnarray*}
   we obtain \eqref{thm_alternation3}.\

   If we neglect $D,$ which usually is assumed to have small entries, we obtain in \eqref{thm_alternation3} that the right-hand side is the Schur complement $\Sigma_1=\widetilde C/C$. For $D \neq 0$ we can give lower and upper bounds: by assumption, $C$ is Hermitian positive definite, and hence has a square root $C^{1/2}$ with inverse $C^{-1/2}$. Introducing the Hermitian positive definite matrix $E=C^{-1/2} D^2 C^{-1/2}$, we have for all integers $m\geq 0$ that
   \begin{eqnarray*} &&
        (-1)^m \Bigl( (I+E)^{-1} - \sum_{j=0}^{m-1} (-1)^j E^j \Bigr) = E^m (I+E)^{-1}
         \\&&= ((E^{1/2})^m)^*(I+E)^{-1} (E^{1/2})^m \geq 0 ,
   \end{eqnarray*}
   where again we write $A \leq B$ for two Hermitian matrices if $B-A$ is positive definite. Substituting gives
   \begin{eqnarray*} &&
       (-1)^m \Bigl( (D+C)^{-1} - \sum_{j=0}^{m-1} C^{-1} (D^2C^{-1})^j \Bigr) \geq 0 ,
   \end{eqnarray*}
   which together with \eqref{thm_alternation3} implies \eqref{thm_alternation1}.
%
\end{proof}

\section{Cosine asymptotics in the univariate case}\label{section_ratio_asymptotics}

In this section we consider orthogonal polynomials of a single complex variable $z\in \mathbb C$. For obtaining asymptotics it is natural to assume throughout the section that $\supp(\mu)$ is infinite. Thus we have the nullspace $\mathcal N(\mu)=0$, $k_n^\mu=n$, and $\mathcal S(\mu)=\mathbb C$, compare with \S\ref{section_def_multivariate}.

In the case of a single complex variable, we want to show that ratio asymptotics is sufficient to derive asymptotics of the cosine $C_n^{\mu}(z,w)$ and thus of $K_n^{\mu+\sigma}(z,z)/K_n^\mu(z,z)$. Our main result
is the following.

\begin{thm}\label{thm_ratio}
    Let $\Omega$ denote a subdomain of the unbounded component of $\mathbb C \setminus \supp(\mu)$, and suppose that there is a function $g$ analytic and different from zero in $\Omega$ such that \begin{equation}\label{thm_ratio0}
        \lim_{n\to \infty} \frac{p_{n}^\mu(z)}{p_{n+1}^\mu(z)} = g(z)  
   \end{equation}
   uniformly on any compact subset of $\Omega$. Let $F$ be a compact subset of $\Omega$.
   Then $1/K^\mu_n(z,z) \to 0$ uniformly for $z\in F$,
   and \begin{equation} \label{thm_ratio1}
          \lim_{n\to \infty} C_n^\mu(z,w) \frac{|p_n^\mu(z)|}{p_n^\mu(z)} \frac{p_n^\mu(w)}{|p_n^\mu(w)|} = \frac{\sqrt{(1-|g(z)|^2)(1-|g(w)|^2)}}{1-g(z)\overline{g(w)}}
   \end{equation}
   uniformly for $z,w\in F$.\

   Furthermore, let $z_1,...,z_\ell,w_1,...w_\ell\in \Omega$ with distinct
   $g(z_1),...,g(z_\ell)$ and distinct $g(w_1),...,g(w_\ell)$. Then, uniformly for $z,w\in F$, there holds
   \begin{eqnarray} && \label{thm_ratio2}
       \lim_{n\to \infty} \Bigl|
       \frac{\det C_n^\mu(z_1,...,z_\ell,z;w_1,...,w_\ell,w)}{\det C_n^\mu(z_1,...,z_\ell;w_1,...,w_\ell)}\Bigr|
       \\&&= \nonumber
       \frac{\sqrt{(1-|g(z)|^2)(1-|g(w)|^2)}}{|1-g(z)\overline{g(w)}|}
       \prod_{j=1}^\ell \Bigl|
       \frac{g(w)- g(w_j)}{1-g(w_j)\overline{g(w)}}
       \frac{g(z)- g(z_j)}{1-g(z)\overline{g(z_j)}}
       \Bigr|.
   \end{eqnarray}
\end{thm}
Before giving the proof of this theorem we mention a sufficient condition for which the hypotheses hold
(see \cite[Proposition 3.4]{Sa10}) .

\begin{lem}\label{lem:rel-asy}
If there exists a function $G(z)$ analytic and non-zero at infinity such that
$$
\lim_{n \to \infty}\frac{zp_n^{\mu}(z)}{p_{n+1}^{\mu}(z)}=G(z),
$$
uniformly in some neighborhood of infinity, then
\begin{equation}\label{lem-ratio}
\lim_{n \to \infty}\frac{p_n^{\mu}(z)}{p_{n+1}^{\mu}(z)}=g(z):=\frac{G(z)}{z},
\end{equation}
uniformly on every closed subset of $\Omega:=\mathbb{C} \setminus \mathrm{Co}(\mathrm{supp}(\mu))$, where $\mathrm{Co}(S)$ denotes the convex hull of the set $S.$
Moreover, $0<|g(z)|<~1$ for $z \in \Omega$.
\end{lem}
\begin{rem}
While it is possible for the limit \eqref{lem-ratio} to hold everywhere in
$\mathbb{C} \setminus \mathrm{supp}(\mu),$ it cannot hold locally uniformly with a limit
 of modulus less than one in any domain containing a point of $\mathrm{supp}(\mu)$.
\end{rem}

\noindent \emph{Proof of Theorem~\ref{thm_ratio}}.
   We may rewrite condition \eqref{thm_ratio0} in the following
   form which will be used in what follows in the proof: for all $\epsilon>0$ there exists an $N=N(\epsilon,F)$ such that,
   for all $n \geq N-1$ and $z\in F$,
   \begin{equation} \label{assumption_ratio}
          |{p_{n}^\mu(z)-g(z)p_{n+1}^\mu(z)}|
          \leq \epsilon \,  {\sqrt{|p_{n}^\mu(z)|^2+|p_{n+1}^\mu(z)|^2}} .
   \end{equation}
   Thus, for $z,w\in F$ and $n \geq N-1$
   \begin{eqnarray} && \label{eq.ratio11}
          |{p_{n}^\mu(z)\overline{p_{n}^\mu(w)}-
          g(z)p_{n+1}^\mu(z)\overline{g(w)p_{n+1}^\mu(w)}}|
          \\&\leq& \nonumber
          |(p_{n}^\mu(z)-g(z)p_{n+1}^\mu(z))\overline{p_{n}^\mu(w)}|
          \\&& \nonumber
          +           |\overline{(p_{n}^\mu(w)-g(w)p_{n+1}^\mu(w))} (p_n^\mu(z)-g(z) p_{n+1}^\mu(z))|
          \\&& \nonumber
          +           |\overline{(p_{n}^\mu(w)-g(w)p_{n+1}^\mu(w))} p_{n}^\mu(z)|
         \\&\leq &  \nonumber
       (2\epsilon +\epsilon^2) \,  {\sqrt{|p_{n}^\mu(z)|^2+|p_{n+1}^\mu(z)|^2}} \,  {\sqrt{|p_{n}^\mu(w)|^2+|p_{n+1}^\mu(w)|^2}} .
   \end{eqnarray}
   For $z=w$ we deduce for sufficiently small $\epsilon$ that
   $$
   \frac{|g(z)|^2 - 2\epsilon - \epsilon^2}{1+ 2\epsilon + \epsilon^2} |p_{n+1}^\mu(z)|^2  \leq |p_{n}^\mu(z)|^2 \leq
   \frac{|g(z)|^2 + 2\epsilon + \epsilon^2}{1- 2\epsilon - \epsilon^2} |p_{n+1}^\mu(z)|^2 ;
   $$
   that is, by assumption on $g$ there are constants $\widetilde q,q >0$ such that, for all $n \geq N$ and $z\in F$,
   $$
             \widetilde q \, | p_{n+1}^\mu(z) | \leq | p_n^\mu(z)| \leq
             q \, | p_{n+1}^\mu(z) |  ,
   $$
   implying that
   \begin{equation} \label{geometric}
            \widetilde q^{n-N} \, | p_{n}^\mu(z) | \leq | p_N^\mu(z)| \leq
             q^{n-N} \, | p_{n}^\mu(z) |  .
    \end{equation}
    By the assumptions on $\Omega$ and $F$ we know from \cite[Theorem~1]{Ambromadze} that
    $$
        \limsup_{n \to \infty} |p_n^\mu(z)|^{1/n} > 1,
    $$
    which together with \eqref{geometric} and $\widetilde q\neq 0$ implies that $p_N^\mu(z)\neq 0$ for $z\in F$, and hence
   $$
         c_1 := \min_{z\in F} |p_N^\mu(z)|^2 > 0 , \quad
         c_2 := \max_{z\in F} K_N^\mu(z,z) <\infty.
   $$
   In addition, taking $n$th roots in \eqref{geometric}, we find that $q<1$. (The same argument  implies that $|g(z)|<1$ for $z\in \Omega$.)
   In particular,
   \begin{eqnarray*}
         K_{n}^\mu(z,z) \geq |p_n^\mu(z)|^2 \geq q^{2(N-n)} |p_N^\mu(z)|^2 \geq c_1 q^{2(N-n)}, 
   \end{eqnarray*}
   showing that $1/K_n^\mu(z,z) \to 0$ uniformly on $F$, as claimed in the theorem.
   Conversely,
   \begin{eqnarray*}
        K_{n}^\mu(z,z) &=& K_{N}^\mu(z,z) + \sum_{j=N+1}^n |p_j^\mu(z)|^2
        \\&\leq&
        c_2 + |p_n^\mu(z)|^2 \sum_{j=0}^n q^{n-j} \leq c_2 + \frac{|p_n^\mu(z)|^2}{1-q}.
   \end{eqnarray*}
   As a consequence,
   \begin{eqnarray*} &&
        |p_N^\mu(z)|\leq \sqrt{K_N^\mu(z,z)} = o(|p^\mu_n(z)|)_{n\to \infty}, \quad
        \sqrt{K_n^\mu(z,z)} = \mathcal O(|p^\mu_n(z)|)_{n\to \infty}, \quad
   \end{eqnarray*}
   uniformly on $F$.

   Summing the inequality \eqref{eq.ratio11} for $j=N,N+1,...,n-1$ we obtain with help of the Cauchy-Schwarz inequality that
   \begin{eqnarray*} &&
       \Bigl| (1-g(z)\overline{g(w)}) \Bigl( K_n^\mu(z,w) - K_N^\mu(z,w) \Bigr) - p^\mu_n(z)\overline{p^\mu_n(w)} + p^\mu_N(z)\overline{p^\mu_N(w)} \Bigr|
       \\&&\leq \sum_{j=N}^{n-1}
          |{p_{j}^\mu(z)\overline{p_{j}^\mu(w)}-
          g(z)p_{j+1}^\mu(z)\overline{g(w)p_{j+1}^\mu(w)}}|
       \\&& \leq (2 \epsilon +\epsilon^2)\, \sum_{j=N}^{n-1} \,  {\sqrt{|p_{j}^\mu(z)|^2+|p_{j+1}^\mu(z)|^2}} \,  {\sqrt{|p_{j}^\mu(w)|^2+|p_{j+1}^\mu(w)|^2}}
       \\&& \leq 2 \, (2 \epsilon +\epsilon^2) \, \sqrt{K_n^\mu(z,z)K_n^\mu(w,w)}.
   \end{eqnarray*}
   Since $\epsilon>0$ is arbitrary, we conclude that
   \begin{eqnarray} && \label{eq.ratio_result1}
        (1-g(z)\overline{g(w)}) K_n^\mu(z,w) = p^\mu_n(z)\overline{p^\mu_n(w)} (1+o(1)_{n\to \infty}),
        \\&& \label{eq.ratio_result2}
        (1-|g(z)|^2) K_n^\mu(z,z) = |p^\mu_n(z)|^2 (1+o(1)_{n\to \infty}),
   \end{eqnarray}
   uniformly on $F$. This implies \eqref{thm_ratio1}.

   For a proof of \eqref{thm_ratio2}, we may suppose that $z_1,...,z_\ell,w_1,...,w_\ell\in F$, and neglect the factors of modulus $1$ on the left-hand side of \eqref{thm_ratio1}. Since the square roots on the right-hand side of \eqref{thm_ratio1} can be factored by linearity of the determinant, for establishing \eqref{thm_ratio2} it is sufficient to recall the well-known expression for the determinant of a Pick matrix, namely
   \begin{equation}\label{pick}
        \det \Bigl( \frac{1}{1-a_j\overline{b_k}}\Bigr)_{j,k=1,...,\ell}
        = \frac{\prod_{j,k=1,j<k}^\ell(a_k-a_j)(\overline{b_k}-\overline{b_j})}
        {\prod_{j,k=1}^\ell(1-a_j\overline{b_k})}.
   \end{equation}
   This last formula follows by considering the left-hand determinant as a function of $a_\ell$, having simple roots at $a_1,...,a_{\ell-1}$ and simple poles at
   $1/\overline{b_1},...,1/\overline{b_\ell}$ (and similarly as a function of $\overline{b_\ell}$). Hence
   $$
        \det \Bigl( \frac{1}{1-a_j\overline{b_k}}\Bigr)_{j,k=1,...,\ell}
        = \frac{C}{1-a_\ell \overline{b_\ell}} \prod_{j=1}^{\ell-1}
        \frac{a_\ell-a_j}{1-a_\ell \overline{b_j}}
        \frac{\overline{b_\ell}-\overline{b_j}}{1-a_j \overline{b_\ell}},
   $$
   with a constant $C$ independent of $a_\ell$. Since the residual at $a_\ell=1/\overline{b_\ell}$ of the right-hand product equals $1$ we obtain, by recurrence on $\ell,$ formula \eqref{pick}.
$\square$\\

 Combining the asymptotic findings of Theorem~\ref{thm_ratio} with our upper and lower bounds \eqref{thm_alternation11} and \eqref{thm_alternation12} for $K_n^{\mu+\sigma}(z,z)/K_n^{\mu}(z,z)$ we prove the following result for adding $\ell$ distinct point masses from $\Omega$; that is, $\sigma=\sum_{j=1}^\ell t_j \delta_{z_j}$.

\begin{cor}\label{cor_ratio}
   Under the assumptions of Theorem~\ref{thm_ratio}, and in particular $z_1,...,z_\ell \in \Omega$ with distinct $g(z_1),...,g(z_\ell)$, we have uniformly on compact subsets of $\Omega$ 
   \begin{equation} \label{cor_ratio1}
       \lim_{n\to \infty} \frac{K_n^{\mu+\sigma}(z,z)}{K_n^{\mu}(z,z)}
        = \Bigl| \prod_{j=1}^\ell \frac{g(z)-g(z_j)}{1-g(z)\overline{g(z_j)}} \Bigr|^2
      ,
   \end{equation}
   and, at a point mass $z=z_m,$
   \begin{equation} \label{cor_ratio2}
        K_n^{\mu+\sigma}(z_m,z_m) = \frac{1}{t_m} \,  \left( 1 + \mathcal O\left(\frac{1}{K_n^\mu(z_m,z_m)}\right)\right)_{n\to \infty}.
    \end{equation}
\end{cor}
\begin{proof}
   We appeal to a special case of Theorem~\ref{thm_alternation} that asserts
   $$
         \Sigma_1 \leq \Sigma_3 \leq \frac{K_n^{\mu+\sigma}(z,z)}{K_n^{\mu}(z,z)} \leq \Sigma_2,
   $$ where we recall that tacitly all quantities depend on $n$ and $z$.  We will also use the abbreviations
   \begin{eqnarray*} &&
      B(z) := \prod_{j=1}^\ell \frac{g(z)-g(z_j)}{1-g(z)\overline{g(z_j)}} , \quad
      B_j(z):= \prod_{k=1,k\neq j}^\ell \frac{g(z)-g(z_k)}{1-g(z)\overline{g(z_k)}}.
   \end{eqnarray*}
   The matrix $C=C_n^\mu(z_1,...,z_\ell;z_1,...,z_\ell)$  is positive semi-definite; hence by \eqref{thm_alternation11}
   \begin{eqnarray*}
      \Sigma_1 &=& \frac{\det C_n^\mu(z_1,...,z_\ell,z;z_1,...,z_\ell,z)}{\det C_n^\mu(z_1,...,z_\ell;z_1,...,z_\ell)}
      \\&=&
       \Bigl|\frac{\det C_n^\mu(z_1,...,z_\ell,z;z_1,...,z_\ell,z)}{\det C_n^\mu(z_1,...,z_\ell;z_1,...,z_\ell)}\Bigr| , 
   \end{eqnarray*}
   which, by \eqref{thm_ratio2} for $w_j=z_j$ and $z=w,$ tends to $|B(z)|^2$ as $n\to \infty$ for $z\in \Omega$. It remains to examine $\Sigma_2-\Sigma_1$: we recall from \eqref{thm_alternation13} that
   $$
       \Sigma_2-\Sigma_1 = \sum_{j=1}^\ell
       \frac{| \det C_n^\mu(z_1,...,z_{j-1},z,z_{j+1},...,z_\ell;z_1,...,z_\ell)|^2}{| \det C_n^\mu(z_1,...,z_\ell;z_1,...,z_\ell)|^2}
       \frac{1}{t_j K_n^\mu(z_j,z_j)} ,
   $$
   which according to \eqref{thm_ratio2} behaves like
   $$
       \sum_{j=1}^\ell
       \frac{1}{t_j K_n^\mu(z_j,z_j)} \Bigl( 1 - \Bigl| \frac{g(z)-g(z_j)}{1-g(z)\overline{g(z_j)}} \Bigr|^2\Bigr) \, \Bigl| \frac{B_j(z)B_j(z_j)}{B_j(z_j)B_j(z_j)} \Bigr|^2,
   $$
   where $B_j(z_j)\neq 0$ by assumption on $g(z_1),...,g(z_\ell)$. From Theorem~\ref{thm_ratio} and its proof we know that $1/K_n^\mu(z_j,z_j)$ exponentially decays to $0$ for all $j$  as $n \to \infty$, implying that \eqref{cor_ratio1} holds.

   Notice that \eqref{cor_ratio1} is not very useful at a point mass $z=z_m$ since then $B(z_m)=0$ and even $\Sigma_1=0$, implying that $b^* = C e_m$, the $m$th column of $C$. Here it is more helpful to return to the definition of $\Sigma_j$ and observe that
   \begin{eqnarray*} &&
        \Sigma_2-\Sigma_1 = b C^{-1} D^2 C^{-1} b^* = e_m^* D^2 e_m = \frac{1}{t_m K_n^\mu(z_m,z_m)} , \quad
       \\&&
       \Sigma_2-\Sigma_3 = b C^{-1} D^2 C^{-1} D^2 C^{-1} b^* = \frac{e_m^* C^{-1} e_m}{\Bigl( t_m K_n^\mu(z_m,z_m) \Bigr)^2} .
   \end{eqnarray*}
   From Cramer's rule and \eqref{thm_ratio2} we deduce that $e_m^* C^{-1} e_m$ has a limit different from $0$ as $n\to \infty$.
   Hence \eqref{cor_ratio2}  holds.
\end{proof}

\begin{rem}
  We have precise asymptotics with error terms in \eqref{cor_ratio2}, but such error terms are missing in \eqref{cor_ratio1} as well as in \eqref{thm_ratio1} and \eqref{thm_ratio2}.

  Under the assumption of Theorem~\ref{thm_ratio}, a slightly more careful error analysis shows that the maximal error for $z\in F$ in \eqref{eq.ratio_result1}, \eqref{eq.ratio_result2} and thus in \eqref{thm_ratio1} is bounded by a constant times
  $$
        \max_{z\in F}
      \Bigl| \frac{p_{n}^\mu(z)-g(z)p_{n+1}^\mu(z)}
                                {\sqrt{|p_{n}^\mu(z)|^2+|p_{n+1}^\mu(z)|^2}} \Bigr|
  $$ plus an exponentially decreasing term. Since $|C_n^\mu(\cdot,\cdot)|\leq 1$, the same is true for \eqref{thm_ratio2} by linearity of the determinant, and thus also for \eqref{cor_ratio1}.
\end{rem}

\begin{rem}
  In the examples presented in the next section, $\supp(\mu)$ is compact with smooth boundary, $\Omega$ is the unbounded connected component of $\mathbb C \setminus \supp(\mu)$ being supposed to be simply connected, and \eqref{thm_ratio0} holds with $g(z)=1/\Phi(z)$, with $\Phi$ the Riemann conformal map from $\Omega$ onto the exterior of the closed unit disk. In particular, $g$ is injective and, with $z_1,...,z_\ell$, also $g(z_1),...,g(z_\ell)$ are distinct. Moreover, the limit in \eqref{cor_ratio1} is different from $0$ for $z\in  \Omega$ different from a point mass.

  Also, in this case, $\Omega$ is known to be regular with respect to the Dirichlet problem, and the measures $\mu$ of \S\ref{sec_Bergman} are all of class {\bf Reg}, which together with \eqref{def_kernel4'} allows us to specify the rate of geometric convergence mentioned in the preceding Remark.
\end{rem}


\section{Asymptotics in Bergman space}\label{sec_Bergman}

In this section we discuss a special and well known case of orthogonality in one complex variable where the precise asymptotic behavior of the Christoffel-Darboux kernels, both  in (a small neighborhood of) the support of the measure of orthogonality and far enough from the support is precisely known; Theorem~\ref{thm_Bergman} below contains precise ratio asymptotics, relevant for our study. The findings of the previous sections are then put to work, yielding the performance of the leverage score, see Corollary~\ref{cor2_Bergman}.

Throughout this section $G$ denotes a bounded open subset of the complex plane with simply connected complement, with boundary $\Gamma$; $\mu$ stands for the area measure on $\mbox{Clos}(G)=G \cup \Gamma$. We denote by $\Phi$ the Riemann outer conformal map from $\mathbb C\setminus \mbox{Clos}(G)$ onto $\mathbb C \setminus \mbox{Clos}(\mathbb D)$ fixing the point at infinity.
As usual, for $r>1$, the compact level sets are defined by their complement,
$$
     \mathbb C \setminus G_r:= \{ z \in \mathbb C \setminus \mbox{Clos}(G): |\Phi(z)|>r \}.
$$
Henceforth $c_j$ is used to denote some absolute, strictly positive constants neither depending on $z$ nor on $n$.

Bergman orthogonal polynomials and their kernels have been used quite successfully as building blocks of conformal maps; the long history of their asymptotics is recorded in the monographs \cite{Gabook87}, \cite{Suetin1}; the recent article \cite{BB-NS} deals with the case of piecewise analytic boundaries. To provide some comparison basis, we start by the simplest case of the unit disk.
\begin{example}\label{exa_Bergman}
   To be more explicit, set $\mu=\mu_{\mathbb D}$, the area measure on the unit disk. Since $p_n^{\mu_{\mathbb D}}(w)=\sqrt{(n+1)/\pi} w^n$,  explicit formulas for the Christoffel kernel are at hand; in particular,
   \begin{eqnarray}
   &&\label{exa1_Bergman}
      \max_{w\in \supp(\mu_{\mathbb D})} \, K_n^{\mu_{\mathbb D}}(w,w)=\frac{(n+1)(n+2)}{2\pi}, \quad
   \\ &&\label{exa2_Bergman}
      K_n^{\mu_{\mathbb D}}(w,w) = \frac{n+1}{\pi} \frac{|w|^{2n+2}}{|w|^2-1} (1 + \mathcal O(1/n)_{n\to \infty}) ,
   \end{eqnarray}
   the second relation being true uniformly on compact subsets of $\mathbb C \setminus \supp(\mu_{\mathbb D})$.
\end{example}

The following theorem collects all relevant estimates and asymptotics.

\begin{thm}\label{thm_Bergman}
   Let $\mu$ be area measure on $G$, and suppose that $\Gamma$ is a Jordan curve which is either piecewise analytic without cusps, or otherwise possesses an arc length parametrization with a derivative being $1/2$-H\"older continuous.

   \bigskip {\bf (a)} For all $n \geq 0$ and $z\in G$:
   $$
        K_n^\mu(z,z) \leq \frac{1}{\pi \mbox{$\dist$}(z,\Gamma)^2}.
   $$

   \bigskip {\bf (b)} With $r(n):=\sqrt{1+\tfrac{1}{n+1}}$ the estimates 
   \begin{eqnarray*}
      \frac{1}{e} \max_{z\in G_{r(n)}} K_n^\mu(z,z)
      &\leq& \max_{z\in \supp(\mu)} K_n^\mu(z,z)
      \leq \max_{z\in G_{r(n)}} K_n^\mu(z,z)
      \\&\leq&
      \gamma_n:=      \frac{c_1}{\mbox{$\dist$}(\Gamma,\partial G_{r(n)})^2}.
   \end{eqnarray*}
   hold for all $n \geq 0$.

   \bigskip {\bf (c)} We have
   $$
         K_n^\mu(z,z) = \frac{n+1}{\pi} \frac{|\Phi'(z)|^2}{|\Phi(z)|^2-1}|\Phi(z)|^{2n+2} (1+\mathcal O(\frac{1}{n})_{n\to \infty})
   $$ uniformly on compact subsets of $\mathbb C \setminus \supp(\mu)$.

   \bigskip {\bf (d)} 
   The asymptotics 
   $$
      \frac{p_n^\mu(z)}{p_{n+1}^\mu(z)} =
      \frac{1}{\Phi(z)}          (1+\mathcal O((1/n)_{n\to \infty})
   $$
  is valid uniformly on compact subsets of $\mathbb C \setminus \supp(\mu)$.
\end{thm}
\begin{proof}
   Part (a) is classical (being valid for any domain $G$); it follows, e.g., from \cite[Lemma 1]{Gabook87}. For a proof of part (b), we claim that the first two inequalities are simple consequences of the maximum principle and \eqref{def_kernel3}. Indeed, for any $r \geq 1$,
   $$
      \max_{z\in \supp(\mu)} K_n^\mu(z,z)\leq
      \max_{z\in G_r} K_n^\mu(z) \leq \max_{z\in G_r} \max_{\deg P \leq n} \frac{|P(z)|^2}{\| P \|_{2,\mu}}.
   $$
   Using the maximum principle for $P(z)$ on $G_r$ we find the upper bound
   $$
      \max_{z\in G_r} K_n^\mu(z) \leq
      \max_{\deg P \leq n} \max_{z\in \partial G_r} \frac{|P(z)|^2}{\| P \|_{2,\mu}} \leq \max_{z\in \partial G_r} K_n^\mu(z,z) ,
   $$
   and thus the maximum of $K_n^\mu(z,z)$ is attained at the boundary. Similarly, from the maximum principle for $P(z)/\Phi(z)^n$ on $\mathbb G \setminus \supp(\mu)$ we infer that
   $$
       \max_{z\in \partial G_{r(n)}} K_n^\mu(z,z)
       \leq r(n)^{2n} \max_{\deg P \leq n} \max_{z\in \supp(\mu)} \frac{|P(z)|^2}{\| P \|_{2,\mu}} \leq e \, \max_{z\in \supp(\mu)} K_n^\mu(z,z),
   $$
   showing that the first two inequalities of part (b) are true.
   In order to show the last inequality, we closely follow \cite{BB-NS}, and consider the row vectors
   \begin{eqnarray*} &&
      P:= (\psi'(w) p_j^\mu(z))_{j=0,...,n}, \quad
      Q:= (\sqrt{\frac{j+1}{\pi}} w^j)_{j=0,...,n}, \quad
      \\&&
      F:= (\psi'(w)\frac{F_{j+1}'(z)}{\sqrt{\pi(j+1)}})_{j=0,...,n},
   \end{eqnarray*}
   where $w=\Phi(z)$ is of modulus larger than $1$, $z=\psi(w)$, and $F_n$ is the $n$th Faber polynomial associated to $G$.
   Notice that $\| P \|^2 = K_n^\mu(z,z)/|\Phi'(z)|^2$, whereas
   $$
       \| Q \|^2 = K_n^{\mu_{\mathbb D}}(w,w) \leq \frac{n+1}{\pi} \frac{|w|^{2n+2}}{|w|^2-1}
   $$ gives the $n$th Bergman Christoffel-Darboux kernel for the unit disk discussed in Example~\ref{exa_Bergman}. From \cite[Eqn. (1.5)]{BB-NS} we know that the $k$the component of $F-Q$ is given by the $k$th component of $(\sqrt{(j+1)/\pi}w^{-j-2})_{j=0,1,...}C$, with $C$ the infinite Grunsky matrix having norm strictly less than $1$, and hence
   $$
         \forall |w|>1 : \quad \| F-Q \|^2 \leq \frac{1}{\pi (|w|^2-1)^2}.
   $$
   From \cite[Eqn. (2.1) and Corollary 2.3]{BB-NS} we know that
   there exists an upper triangular matrix $R_n$ with $\| R_n\|\leq 1$ and $\| R_n^{-1} \| \leq 1/\sqrt{1-\|C\|^2}$ such that
   $F=PR_n$. Thus
   $$
         \forall |w|>1: \quad \| F \|^2 \leq \| P \|^2 \leq \frac{\| F \|^2}{1-\| C \|^2}.
   $$
   Combining these findings gives for $z\in \partial G_{r(n)}$ and hence $|w|^2-1=1/(n+1)$
   \begin{eqnarray*}
       K_n^\mu(z,z) = | \Phi'(z)|^2 \| P \|^2 &\leq& | \Phi'(z)|^2 \frac{(\| Q \| + \| F-Q \|)^2}{1-\| C \|^2}
       \\&\leq& \frac{2.5}{1-\| C \|^2} \frac{| \Phi'(z)|^2}{(| \Phi(z)|^2-1)^2}.
   \end{eqnarray*}
   Recalling from \cite[Theorem 3.1]{Trefethen}
   that
   $$
       \frac{1/2}{\mbox{dist}(z,\Gamma)} \leq \frac{|\Phi'(z)|}{|\Phi(z)|-1}
       \leq \frac{2}{\mbox{dist}(z,\Gamma)} ,
   $$
   we have established the last inequality of part (b), with
   $c_1 \leq \tfrac{2.5}{1-\| C \|^2}$ depending only on the geometry of $G$.

   For a proof of parts (c) and (d), we recall from \cite[Eqn. (1.12)]{BB-NS} that $\varepsilon_n = \| C e_n \|^2=\mathcal O(1/n)$ by our assumptions on $\Gamma$, and hence by \cite[Theorem~1.1]{BB-NS}, uniformly for $z$ in a compact subset of $\mathbb C \setminus \supp(\mu)$,
   \begin{equation} \label{eq3_Bergman}
         p_n^\mu(z) = \Phi'(z) \sqrt{n+1}{\pi} \Phi(z)^n \left( 1 + \mathcal O\left(\frac{1}{n+1}\right)_{n \to \infty}\right) .
   \end{equation}
   Then part (c) follows by taking squares in \eqref{eq3_Bergman}, summing, and finally applying \eqref{exa2_Bergman} with $w=\Phi(z)$.
   Also, part (d) on ratio asymptotics is an immediate consequence of \eqref{eq3_Bergman}.
\end{proof}

\begin{rems}\label{rem_Bergman}
  {\bf (a)}
  We claim without proof that the estimates of the previous statement and its proof could be improved if we are willing to add additional smoothness assumptions on the boundary $\Gamma$. E.g., for a $\mathcal C^2$ boundary $\Gamma$, $\| F-Q \|/\| Q \|$ can be shown to tend to zero for $n\to \infty$ uniformly for $|\Phi(z)|^2-1\geq 1/(n+1)$, leading to lower and upper bounds for $K_n^\mu(z,z)$ that are sharp for $n \to \infty$ up to a constant.

  {\bf (b)}
  For general $\Gamma$ as in Theorem~\ref{thm_Bergman},
  the quantity $\gamma_n$ of Theorem~\ref{thm_Bergman}(b) can be shown to behave like $(n+1)^2$ in the case when $\Gamma$ has no corner of outer angle $>\pi$. Since this is true in particular for a $\mathcal C^2$ boundary, our Theorem~\ref{thm_Bergman}(b) is in accordance with a recent result of Totik \cite[Theorem~1.3]{To09}
  who showed under this additional assumption that the limit $$
      \lim_{n\to \infty} \frac{1}{n^2}\max_{z\in \partial \supp(\mu)} K_n^\mu(z,z)
      \quad \mbox{is finite and $\neq 0$.}
  $$
  In the case when there is a maximal outer angle $=\alpha \pi$ with $\alpha>1$, $\gamma_n$ can be shown to behave like $(n+1)^{2\alpha}$, which we believe to be also the rate of growth of $\max_{z\in \partial \supp(\mu)} K_n^\mu(z,z)$. .

  {\bf (c)} The $\mathcal{O}(1/n)$ error term in Theorem~\ref{thm_Bergman}(c) can be improved for smoother boundaries $\Gamma$, but not in Theorem~\ref{thm_Bergman}(d), see Example~\ref{exa_Bergman}.\qed
\end{rems}

In the next statement we improve a statement of Lasserre and Pauwels { \cite[Theorem~3.9 and Theorem~3.12]{LP2};} compare with Remark~\ref{rem_outliers1}. Here we allow for a finite number of fixed outliers in the case of one complex variable.

\begin{cor}\label{cor1_Bergman}
Let $\mu$ be as in Theorem~\ref{thm_Bergman}, and
$$
    \nu=\mu+\sigma, \quad \sigma=\sum_{j=1}^\ell t_j \delta_{z_j} , \quad t_j>0, \quad z_1,...,z_\ell \in \mathbb C \setminus \supp(\mu) \mbox{~distinct.}
$$
Then the Hausdorff distance between $\supp(\nu)$ and the level set $S_n:=\{ z \in \mathbb C : K_n^\nu(z,z)\leq \gamma_n \}$ with $\gamma_n$ as in Theorem~\ref{thm_Bergman}(b) tends to zero as $n\to \infty$.
\end{cor}
\begin{proof}
   We first recall from Theorem~\ref{thm_alternation} that $K_n^\nu(z,z)\leq K_n^\mu(z,z)$, and hence $\supp(\mu)\subset S_n$ by Theorem~\ref{thm_Bergman}(b). Also, setting $\nu_j := \nu-t_j \delta_{z_j}$, we know from Theorem~\ref{thm_alternation} that
   $$
        K_n^\nu(z_j,z_j) =
        \frac{1/t_j}{1+\frac{1}{t_j K_n^{\nu_j}(z_j,z_j)}}.
   $$
   The quantity $K_n^{\nu_j}(z_j,z_j)/K_n^\mu(z_j,z_j)$ has a non-zero limit according to Corollary~\ref{cor_ratio} and Theorem~\ref{thm_Bergman}(d). Also,  Theorem~\ref{thm_Bergman}(c) together with Remark~\ref{rem_Bergman}(b) imply that $K_n^\mu(z_j,z_j)/\gamma_n$ grows geometrically large. Hence $K_n^\nu(z_j,z_j) \to 1/t_j$ with a geometric rate, implying that $\supp(\nu) \subset S_n$ for sufficiently large $n$.

   For $z$ outside of a neighborhood $U$ of $\supp(\nu)$ we write
   $$
       \frac{K_n^{\nu}(z,z)}{\gamma_n} = \frac{K_n^{\nu}(z,z)}{K_n^{\mu}(z,z)} \frac{K_n^{\mu}(z,z)}{\gamma_n}.
   $$ The first factor on the right-hand side has a non-zero limit according to Corollary~\ref{cor_ratio} and Theorem~\ref{thm_Bergman}(d) uniformly for $z\not\in U$. Also, again by Theorem~\ref{thm_Bergman}(c) and Remark~\ref{rem_Bergman}(b), $K_n^\mu(z,z)/\gamma_n$ grows at least as a constant times $|\Phi(z)|^{2n}/n^{\beta}$ uniformly for $z\not\in U$ for some constant $\beta>0$, and we conclude that $S_n \subset U$ for sufficiently large $n$, showing the convergence of the Hausdorff distance.
\end{proof}

 We are now prepared to describe a situation where the efficiency of our leverage score for detecting outliers can be tested.

\begin{cor}\label{cor2_Bergman}
Let $\mu,\nu,\sigma$ be as in Corollary~\ref{cor1_Bergman}, and $n$ be sufficiently large. Furthermore, consider the discrete measures
$$
    \widetilde \nu = \widetilde \mu + \sigma, \quad
    \widetilde \mu = \sum_{j=\ell+1
    }^N t_j \delta_{z_j} , \quad t_j>0, \quad z_{\ell+1},...,z_N \in \supp(\mu) \mbox{~distinct,}
$$
where we assume that $\widetilde \mu$ is sufficiently close to $\mu$ such that
$\| M_n(\widetilde \mu,\mu) - I \| \leq 1/2$. Then there exist $c>0$ and $q\in (0,1)$ depending only on $\mu,\sigma$ but not on $\widetilde \nu$ such that, for the outliers,
$$
   j=1,2,...,\ell: \quad
   1 - t_j K_n^{\widetilde \nu}(z_j,z_j) \leq c q^n ,
$$
whereas for the other mass points of $\widetilde \nu$
$$
   j = \ell+1,...,N : \quad
   t_j K_n^{\widetilde \nu}(z_j,z_j) \leq
   \frac{3}{2} \min \Bigl\{ t_j \gamma_n , \frac{t_j}{\pi \mbox{$\dist$}(z_j,\Gamma)^2} \Bigr\} ,
$$
where $\gamma_n$ is as in Theorem~\ref{thm_Bergman}(b).
\end{cor}
\begin{proof}
   We first consider the case $j\in \{ 1,...,\ell\}$ of outliers. Let $\nu_j=\nu - t_j \delta_{z_j}$, $\widetilde \nu_j=\widetilde \nu - t_j \delta_{z_j}$. According to Remark~\ref{eq.Rayleigh_mass} with $\epsilon=1/2$ we have $K_n^{\widetilde \nu_j}(z_j,z_j) \geq K_n^{\nu_j}(z_j,z_j)/2$, and hence
   \begin{eqnarray*}
      1 - t_j K_n^{\widetilde \nu}(z_j,z_j) &=& \frac{1}{1+t_j K_n^{\widetilde \nu_j}(z_j,z_j)}
      \leq \frac{2}{t_j K_n^{\nu_j}(z_j,z_j)}
      \\&=& \frac{2}{t_j}\frac{K_n^{\mu}(z_j,z_j)}{ K_n^{\nu_j}(z_j,z_j)} \frac{1}{K_n^{\mu}(z_j,z_j)},
   \end{eqnarray*}
   and we conclude as in the previous proof using Theorem~\ref{thm_Bergman}(c),(d) and Corollary~\ref{cor_ratio} that our first assertion holds.

   Now we consider the case $j\in \{ \ell+1,...,N \}$ of mass points $z_j \in \supp(\mu)$. Applying first Theorem~\ref{thm_alternation} and then Proposition~\ref{prop.2measures} with $\epsilon=1/2$, we arrive at
   \begin{eqnarray*} &&
       t_j K_n^{\widetilde \nu}(z_j,z_j)
       \leq
       t_j K_n^{\widetilde \mu}(z_j,z_j)
       \leq
       \frac{3}{2} t_j K_n^{\mu}(z_j,z_j).
   \end{eqnarray*}
   Thus our second assertion is a consequence of Theorem~\ref{thm_Bergman}(a),(b).
\end{proof}

By comparing the upper left entry,
our assumption $\| M_n(\widetilde \mu,\mu) - I \| \leq 1/2$ also implies that
$$
     \frac{\widetilde \mu(\mathbb C)}{\mu(\mathbb C)} =
     \frac{1}{\mu(\mathbb C)} \sum_{j=\ell+1}^N t_j \in [1/2,3/2],
$$
and thus a typical weight of $\widetilde \mu$ is of order $1/N$.  Hence, roughly speaking, our leverage score $t_j K_n^{\widetilde \nu}(z_j,z_j)$ is very close to $1$ for the case $j \in \{ 1,...,\ell\}$ of outliers, and at least bounded for $z_j$ lying in the interior of $\supp(\mu)$, and more precisely, having a weight of typical size and $z_j$ being at least of distance $1/\sqrt{N}$ to the boundary $\Gamma$. Moreover, for equal weights $t_{\ell+1}=\cdots=t_N$ (and thus of order $1/N$), as long as
\begin{equation} \label{Bergman_assumption}
   \mbox{$\ell$ is fixed, $n$ is sufficiently large, and $N \geq \gamma_n$}
\end{equation}
Corollary~\ref{cor2_Bergman} tells us that we are able to identify correctly the outliers as those elements $z_j$ of the support where our leverage score $t_j K_n^{\widetilde \nu}(z_j,z_j)$ is close to $1$.

In order to illustrate this claim, we conclude this section by reporting some numerical simulations. Here we discuss for different clouds of distinct points $z_1,...,z_N\in \mathbb C$, and equal weights $t_1=\cdots=t_N=1/N$, with the corresponding discrete measure $\nu$. The points $z_j$ of the clouds are drawn following a color code given by our leverage score $t_j K_n^\mu(z_j,z_j)$, with red corresponding to values close to $0$, and blue values close to $1$. For each cloud, we draw in the upper row (without theoretical justification) the level scores for bivariate orthogonal polynomials with lexicographical ordering for total degrees $1,4,8$ (and hence $n=2,14,44$) and in the lower row for univariate Bergman orthogonal polynomials for $n=1,9,44$. The first column essentially corresponds to the classical leverage score known from data analysis and mentioned in Section~\ref{sec_leverage}.

The vector of values of the leverage score has been computed in finite precision arithmetic as follows: in the univariate case, we first compute by the full Arnoldi method in complexity $\mathcal O(n^2N)$ the Hessenberg matrix allowing to represent $zv_n^{\widetilde \nu}(z)$ in terms of $v_{n+1}^{\widetilde \nu}(z)$, and then use a link between values of the Christoffel-Darboux kernel and GMRES for the shifted Hessenberg matrix, with a total complexity of $\mathcal O(n^3N)$. A generalization of this approach has been used for bivariate orthogonal polynomials. There exist more efficient approaches, but many of them suffer from loss of orthogonality. For each of our simulations we give an indicator showing that our approach does not have this drawback. Also, and this is probably {\it the most important message for large data sets, the complexity and memory requirements scale linearly with $N$}.

\begin{figure}
   \centerline{\includegraphics[scale=0.6]{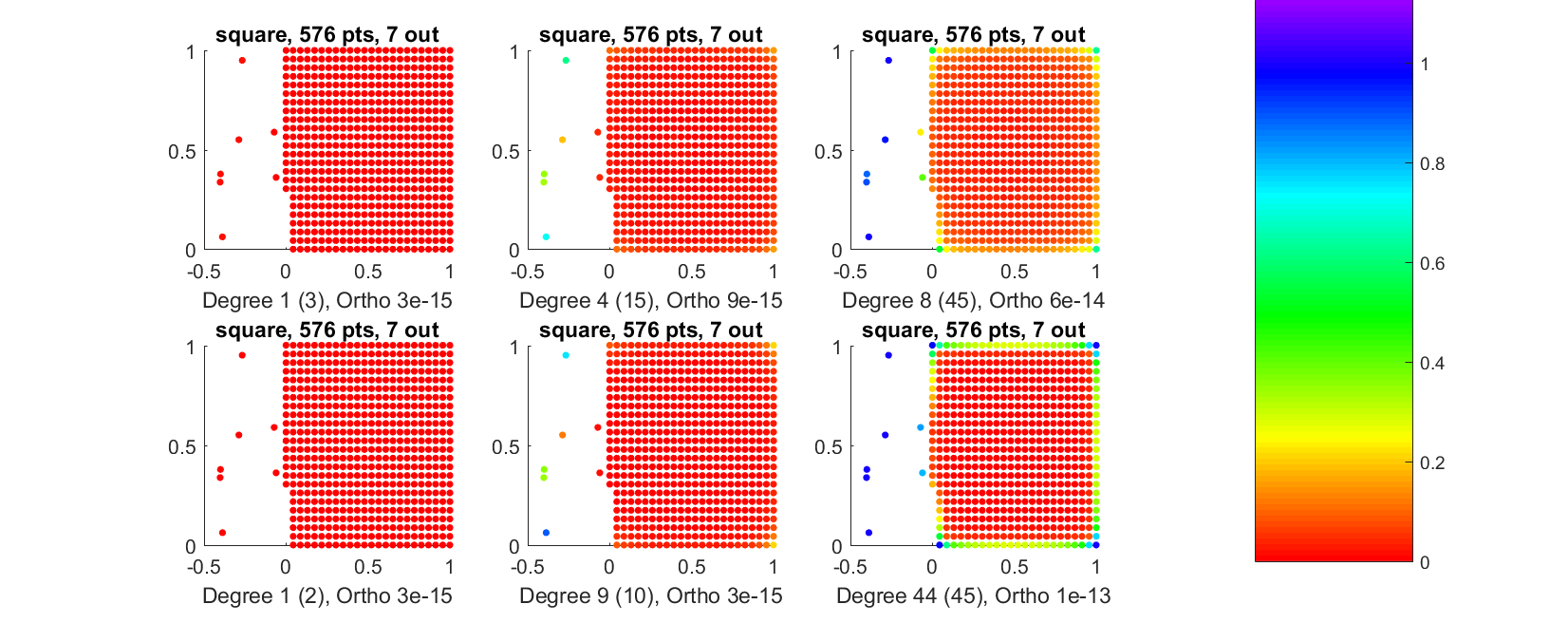}}
   \caption{A cloud of $N=576$ points, with $7$ random outliers, the others obtained by discretizing normalized Lebesgue measure on the unit square by a regular grid.}\label{fig71}
\end{figure}

\begin{figure}
   \centerline{\includegraphics[scale=0.6]{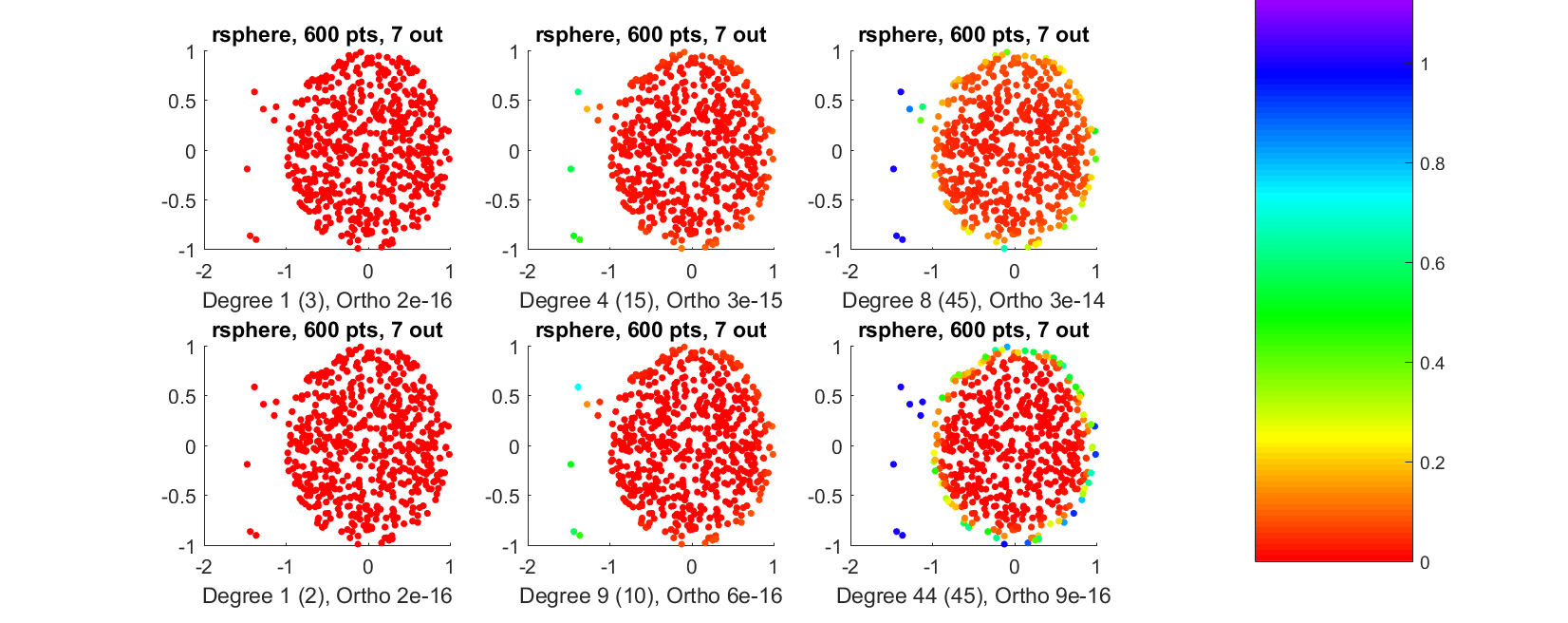}}
   \caption{A cloud of $N=600$ points, with $\ell=7$ random outliers. The other points are random samplings of the unit disk.}\label{fig72}
\end{figure}

\begin{figure}
   \centerline{\includegraphics[scale=0.6]{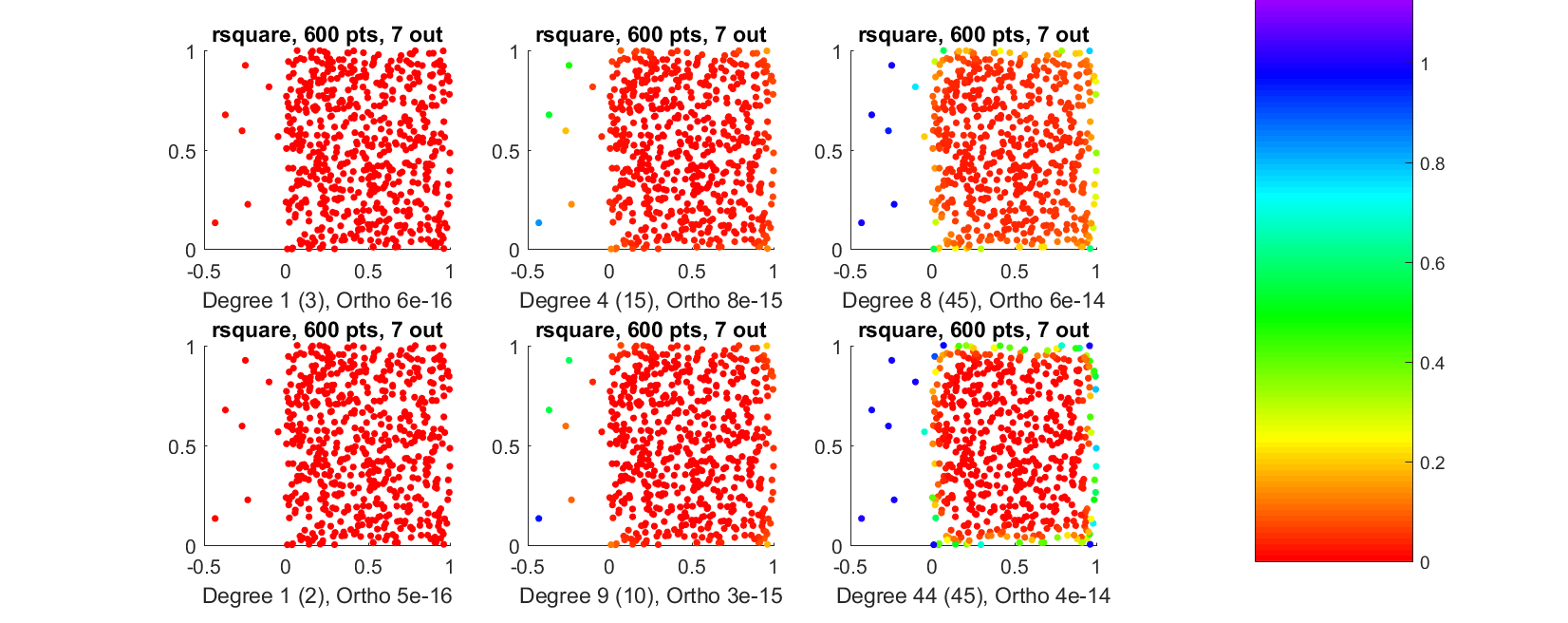}}
   \centerline{\includegraphics[scale=0.6]{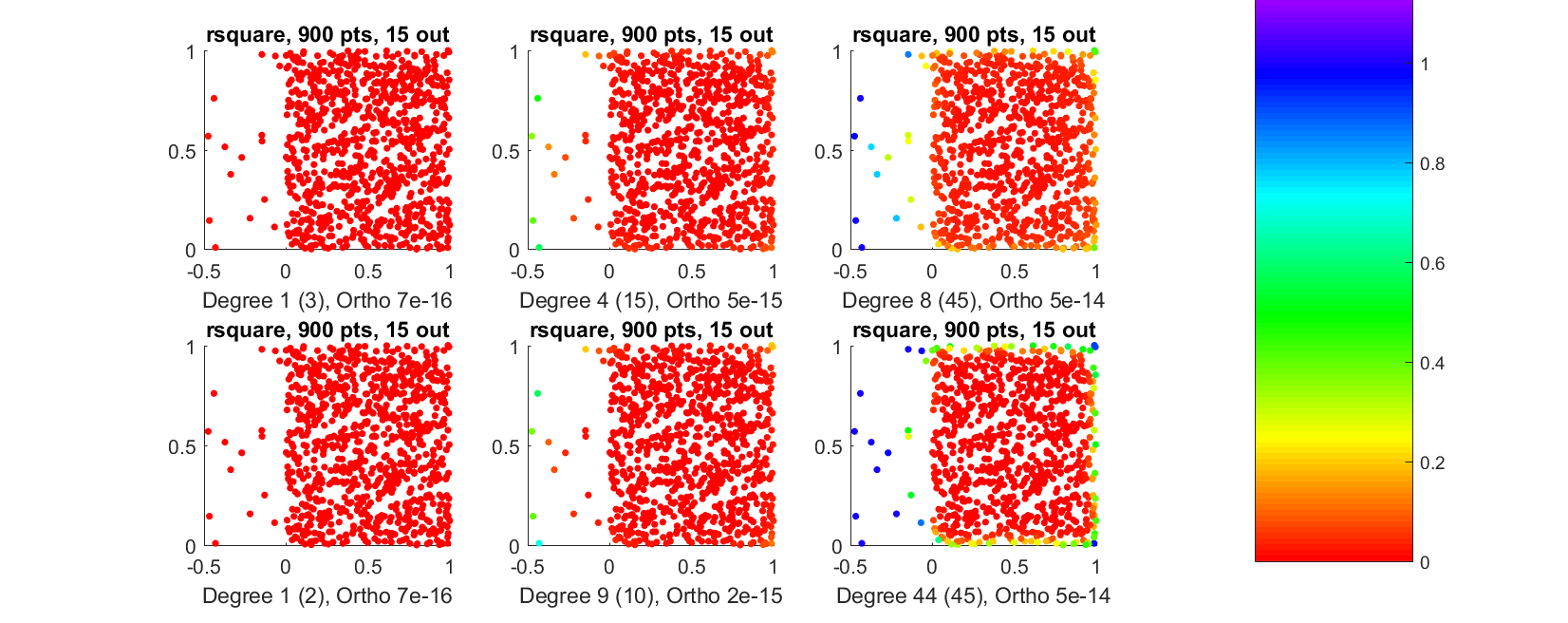}}
   \caption{A first cloud of $N=600$ points, with $\ell=7$ random outliers, and a second one with $N=900$ and $15$ outliers. The other points are random samplings.}\label{fig73}
\end{figure}

In the numerical experiments displayed in Figures~\ref{fig71}, \ref{fig72}, and \ref{fig73}, we observe that the classical leverage scores in the first column do not detect any of our outliers, probably due to the small weights $t_j=1/N$. In the experiences of the other two columns, we clearly detect outliers, but the color separation is more pronounced for one complex variable, in particular for outliers which are closer to $G$. In the right column for $n=44$, the color code for the outliers is best, but in fact also corners of the domain and some other points of the support of $\widetilde \nu$ have an outlier color code. This seems to be partly a consequence of the sampling procedure, but clearly also a consequence of the fact that we do not respect \eqref{Bergman_assumption} since $N$ is not large enough.

Similar sharp estimates are known for measures supported on Jordan curves, or arcs in the complex plane, but we do not detail them here.

\clearpage
\section{Examples in the multivariate case}\label{section_examples_multivariate}
The present section collects a series of mostly known facts of pluripotential theory which are relevant for our study. In \S~\ref{subsection_Green} we recall several well-known formulas for the plurisubharmonic Green function $g_\Omega(z)$. In \S~\ref{sec_tensor} we consider the particular case of tensor measures on the unit square allowing to get far more precise results on asymptotics. We also include a simple numerical experiment endorsing the thesis that working in the complex domain has clear benefits even for problems formulated
solely in real variables. Finally, in \S~\ref{sec_Reinhardt} we discuss two other cases of several complex variables where explicit formulas for the Christoffel-Darboux kernel are available.

\subsection{Examples of plurisubharmonic Green functions with pole at infinity}\label{subsection_Green}

 With respect to and support of Lemma \ref{lem_growth}, we are fortunate to have closed form expressions for the plurisubharmonic Green function $g_\Omega(z)$
with pole at infinity for a few classes of compact subsets

   $S=\mathbb C\setminus \Omega$ of $\C^d$. For all these examples, $S$ is non-pluripolar, $L$-regular and polynomially convex, and hence the explicit formulas for $g_\Omega(z)$ are obtained via computing Siciak's extremal function, see \cite{Klimek}.
\color{black}
Notice also the affine invariance not only of the Christoffel-Darboux kernel \eqref{affine_invariance}, but also of the plurisubharmonic Green function \cite[\S 5.3]{Klimek}, and thus a property being for instance valid for the real unit ball also translates to any scaled and shifted real ball. While most of results were originally stated for the total degree, and the underlying lexicographical order, a new trend of considering
degrees defined by the Minkowski functional of a convex subset of the multi-indices is becoming popular \cite{Bayraktar,BosLevenberg_new}.

\subsubsection{Complex ball}

 Consider for instance a complex norm $[ \cdot ]$ in $\C^d$ and the closed ball
$$ B_a(r) = \{ z \in \C^d; \ [z-a] \leq r\}.$$
By a complex norm we mean a norm with the homogeneity property
$$ [\lambda z] = |\lambda| [z], \ \ \lambda \in \C, \ z \in \C^d.$$
Let $\Omega$ denote the complement of $B_a(r)$ in $\C^d$. Then the set $B_a(r)$ turns out to be regular and
\begin{equation}\label{complex-ball}
g_\Omega(z) = \log^+ \frac{[z-a]}{r}.
\end{equation}
For a proof, see Example 5.1.1 in \cite{Klimek}. In particular, for the unit ball $\B^d$ with respect to the standard Hermitian norm $\| \cdot \|$ we obtain
$$ g_{\C^d \setminus \B^d}(z) = \log^+ \| z \|.$$

\subsubsection{Polynomial polyhedra} Let $p_1, p_2, \ldots, p_d$ be complex polynomials in $d$ variables, with highest homogeneous parts $\hat{p}_1, \hat{p}_2, \ldots, \hat{p}_d$
respectively. Assume that
$$ \sum_{j=1}^d |\hat{p}_j (z)| = 0 \  \ {\rm iff} \ \ z =0;$$
that is, the map $(p_1,p_2,\ldots,p_d) : \C^d \longrightarrow \C^d$ has finite fibres. Consider the analytic polyhedron
$$ K = \{ z \in \C^d; \ |p_j(z)| \leq 1, \ \ 1 \leq j \leq d\},$$
and its complement $\Omega = \C^d \setminus K$. Then (see Corollary 5.3.2 in \cite{Klimek})
\begin{equation}
g_\Omega(z) = \max_{1 \leq j \leq d} \frac{ \log^+ |p_j(z)|}{\deg p_j}.
\end{equation}
 In particular, for the polydisk
$$ \D^d = \{ z \in \C^d; \  |z_j| \leq 1, \  1 \leq j \leq d\}$$
we get
\begin{equation}
g_{\C^d \setminus \D^d}(z) = \max_{1 \leq j \leq d} \log^+ |z_j|.
\end{equation}

\subsubsection{Product domains} A theorem due to Siciak provides the naturality of the Green function on product sets.
More specifically, let $K \subset \C^d$ and $L \subset \C^e$ be compact subsets. Then
\begin{equation}\label{product}
g_{\C^{d+e} \setminus (K \times L)}(z,w) = \max (g_{\C^d \setminus K}(z), g_{\C^e \setminus L}(w)), \ \ z \in \C^d, \ w \in \C^e.
\end{equation}
For details, see Theorem 5.1.8 in \cite{Klimek}.

\subsubsection{Real subsets of complex space} Although a great deal of examples live in $\R^d$, it is necessary to treat them as subsets of $\C^d$.
Almost all closed form known formulas are obtained by pull-back from the Green function with pole at infinity of the real interval $[-1,1]$,
given by
$$ g_{\C \setminus [-1,1]}(z) = \log |z + \sqrt{z^2-1}|, \ z \in \C \setminus [-1,1],
$$
and $g_{\C \setminus [-1,1]}(z)=0$ in for $z\in [-1,1]$. Here the square root is chosen so that $g_{\C \setminus [-1,1]} \geq 0$.

Consider a compact subset $E \subset \R^d$ that is convex, has non-empty interior and is
symmetric with respect to the origin: $x \in E$ implies $-x \in E$. Following Lundin \cite{Lundin} one can represent $E$ as follows
$$ E = \{ z \in \C^d: \ \forall \omega \in S^{d-1}, a(\omega) \omega^\ast z \in [-1,1]\}, $$
where $S^{d-1}$ denotes the unit sphere in $\mathbb R^d$, and $a$ is continuous on $S^{d-1}$. We can choose $a(\omega)$ to be equal to the inverse of half the width of $E$
in the direction $\omega$. If the boundary of $E$ is smooth, then there is no ambiguity in defining $a(\omega)$.
The main result of \cite{Lundin} then states that
$$ g_{\C^d \setminus E}(z) = \max_{\omega \in S^{d-1}} g_{\mathbb C \setminus [-1,1]}(a(\omega) \omega^\ast z).$$

A few particular cases are relevant and we list them under separate subsections.

\subsubsection{Real ball} On the unit ball $B^d \subset \R^d$ we remark $a(\omega) = 1$ for all values of $\omega \in S^{d-1}$
and therefore
\begin{equation}\label{real-ball}
g_{\C^d \setminus B^d}(z) = \frac{1}{2} g_{\C \setminus [-1,1]}(\| z \|^2 + | z \cdot z -1 |), \ \ z \notin B^d.
\end{equation}
This result is due to Siciak, see Theorem 5.4.6 in \cite{Klimek}. In particular, for  a {\it real} vector $x \in \R^d$ we infer:
 $$g_{\mathbb C^d\setminus B}(x) = g_{\mathbb C \setminus [-1,1]}(\| x \|)
     = \log(\| x \| + \sqrt{\| x \|^2-1}) = \frac{1}{2} g_{\mathbb C \setminus [-1,1]}(2\| x \|^2-1).
$$

\subsubsection{Real cube} According to the product formula (\ref{product}) we find
\begin{equation}
g_{\C^d \setminus [-1,1]^d}(z) = \max_{1 \leq j \leq d} g_{\mathbb C \setminus [-1,1]}(z_j)|.
\end{equation}

\subsubsection{Simplex} Denote by $\Delta^d$ the standard simplex in $\R^d$; that is, the convex hull of the
vectors $(0,e_1,\cdots,e_d)$, where $(e_j)$ is the standard orthonormal basis. A base change of Lundin's formula
via the map $(z_1,z_2,\cdots,z_d) \mapsto (z_1^2, z_2^2, \cdots, z_d^2)$ leads to the following closed form expression, discovered by Baran:
\begin{equation}
g_{\C^d \setminus \Delta^d}(z) = g_{\mathbb C \setminus [-1,1]}( |z_1| + |z_2| + \cdots + |z_d| + |z_1 + \cdots + z_d -1|).
\end{equation}
For details, see Example 5.4.7 in \cite{Klimek}.

\subsection{The real square with tensor product of Chebyshev weights}\label{sec_tensor}

Let us consider for $x = (x_1,x_2)^T \in [-1,1]^2$ the tensor measure
$$
     d\mu(x) = d\omega(x_1) d\omega(x_2), \quad d\omega(x_1)=\frac{dx}{\pi \sqrt{1-x_1^2}}.
$$
The orthogonal polynomials for the equilibrium measure $\omega$ on $[-1,1]$ are explicitly known; namely, $p_0^\omega(z)=1$ and $p_n^\omega(z)=\sqrt{2} T_n(z)$ for $n \geq 1$,
where $T_n$ are Chebyshev polynomials of the first kind.
In particular, if $\alpha(n)=(j,k)$ then $p_n^\mu(x)=p_j^\omega(x_1)p_k^\omega(x_2)$, making it possible to get more explicit information about the underlying Christoffel-Darboux kernel.

\begin{thm}\label{thm_tensor}
  Consider the enumeration of the multi-indices consistent with the tensor structure; that is, for all integers $n\geq 0$ with $N=n_{max}(n) =(n+1)^2-1$
  $$
       \{\alpha(0),...,\alpha(n_{max}(N)) \} =
       \left\{
       \left[\begin{array}{cc} j \\ \ell
        \end{array}\right], 0 \leq j,\ell\leq n \right\},
  $$
  and set
  \begin{equation} \label{plot_green2}
     g(z)=g_{\mathbb C\setminus[-1,1]}(z_1)g_{\mathbb C\setminus[-1,1]}(z_2).
  \end{equation}
  {\bf (a)} For $z'=(1,1)$ and  $N=(n+1)^2-1,$
  $$
       \max_{z\in \supp(\mu)} K_N^\mu(z,z) = K_N(z',z')=(2n+1)^2 \leq 4 N+1.
  $$
  {\bf (b)} For all $z\in \mathbb R^2 \setminus \supp(\mu)$ and  $N=(n+1)^2-1\to \infty$,
  $$
              K_N^\mu(z,z)\sim
              \left\{\begin{array}{ll}
                 e^{2ng(z)} & \mbox{in the ``corner case" $|z_1|> 1$ and $|z_2|> 1$,}
                 \\
                 (n+1) e^{2ng(z)} & \mbox{else},
              \end{array}\right.
              ,
  $$
  where $\sim$ means that the quotient has a finite and non-zero limit for $n\to \infty$.
  \\
  {\bf (c)} For all (generic) $z,w\in \mathbb R^2 \setminus \supp(\mu)$ with $z_1\neq w_1$ and $z_2 \neq w_2$, and  $N=(n+1)^2-1\to \infty$, $n$ even, we have $C_N^\mu(z,w)$ tending
  to some finite non-zero limit in the corner case for both $z,w$, and else tending to $0$.
\end{thm}
\begin{proof}
  We first mention that, for $N=(n+1)^2-1$,
  $$
      K_N^\mu(z,w)=K_n^\omega(z_1,w_1)K_n^\omega(z_2,w_2), \, \, \,
      C_N^\mu(z,w)=C_n^\omega(z_1,w_1)C_n^\omega(z_2,w_2),
  $$
  and hence for the assertion of the Theorem it is sufficient to consider the univariate case. For a proof of part (a), it is sufficient to recall that, for $z_1\in [-1,1]$, $T_k(z_1)^2 \leq 1=T_k(1)^2$, and hence $K_n^\omega(z_1,z_1)\leq K_n^\omega(1,1)=2n+1$.

  For a proof of part (b) and (c), we write
  $$
      z_j = \frac{1}{2}( u_j + \frac{1}{u_j}) , \quad
      w_j = \frac{1}{2}( v_j + \frac{1}{v_j}) , \quad
      |u_j| \geq 1, \quad |v_j|\geq 1,
  $$
  and $\log|u_1|=g_{\mathbb C \setminus [-1,1]}(z_1)$.
  Then by using geometric sums it is quite easy to check that
\begin{equation*} \label{kernel_cheby}
   K_n^\omega(z_1,z_1) \sim \left\{
   \begin{array}{ll}
      \frac{|u_1|^{2n}}{2(1-1/|u_1|^2)}
      & \mbox{if $z_1\not\in [-1,1]$, that is, $|u_1|>1$,}
      \\
      (n+1)
      & \mbox{if $z_1=\cos(t)\in [-1,1]$, that is, $u_1=e^{it}$,}
   \end{array}\right.
\end{equation*}
   implying part (b).

   For a proof of part (c), we recall that, for distinct real $w_1,z_1\not\in[-1,1]$ with $|z_1|>1$, $|w_1|> 1$ we obviously have ratio asymptotics $p_n^\omega(z_1)/p_{n+1}^\omega(z_1) \to 1/u_1$ for $n \to \infty$, and hence by Theorem~\ref{thm_ratio}
   $$
        \lim_{n \to \infty} C_n^\omega(z_1,w_1) (\frac{z_1w_1}{|z_1w_1|})^n =
        \frac{\sqrt{(u_1^2-1)(v_1^2-1)}}{u_1v_1 - 1},
   $$
   which simplifies for even $n$.
   If however $z_1\in [-1,1]$, then the quotient $|K_n^\omega(z_1,w_1)|/\sqrt{K_n^\omega(w_1,w_1)}$ is shown to be bounded in both cases $w_1 \not\in [-1,1]$, and $w_1\in [-1,1]\setminus\{ z_1 \}$, and hence $C_n^\omega(z_1,w_1)$ tends to $0$ by part (b).
\end{proof}

If $z\in \mathbb R^2\setminus \supp(\mu)$ and the masspoints of $\sigma$ have distinct coordinates,
we conclude as in the proof of Corollary~\ref{cor_ratio} that $K_N^{\mu+\sigma}(z,z)/K_{N}^{\mu}(z,z)$ has a limit for $N=(n+1)^2-1 \to \infty$, which is different from one only if the corner case holds for $z$ and at least one mass point. Hence, as in Corollaries \ref{cor1_Bergman} and \ref{cor2_Bergman}, we verify that our bivariate leverage score works in this setting.

In order to compare the preceding theorem with Lemma \ref{lem_growth}, we notice that, in case of graded lexicographical ordering and $N=n_{tot}(n)=(n+1)(n+2)/2-1$, we still have that $\max_{z\in \supp(\mu)} K_N^\mu(z,z)\leq 4N+1$. Hence, with the plurisubharmonic Green function of $[-1,1]^2$
\begin{equation} \label{plot_green1}
     \widetilde g(z)=\max\{g_{\mathbb C \setminus [-1,1]}(z_1),g_{\mathbb C \setminus [-1,1]}(z_2) \},
\end{equation}
one may show that
$$
       \frac{1}{2} e^{2n\widetilde g(z)} \leq K_N^\mu(z,z) \leq
       e^{2n\widetilde g(z)} \max_{\deg P \leq N} \max_{x\in \supp(\mu)} \frac{|p(x)|^2}{\| P \|_{2,\mu}^2}
       \leq
       (4N+1)e^{2n\widetilde g(z)} ,
$$
the left-hand side being obtained by \eqref{def_kernel3_ter} for $p(z_1,z_2)\in \{p_n^\omega(z_1),p_n^\omega(z_2) \}$. All these findings are less precise than for the case of tensor ordering, and it seems to be non-trivial to get cosine asymptotics.

\begin{figure}
   \centerline{\includegraphics[scale=0.6]{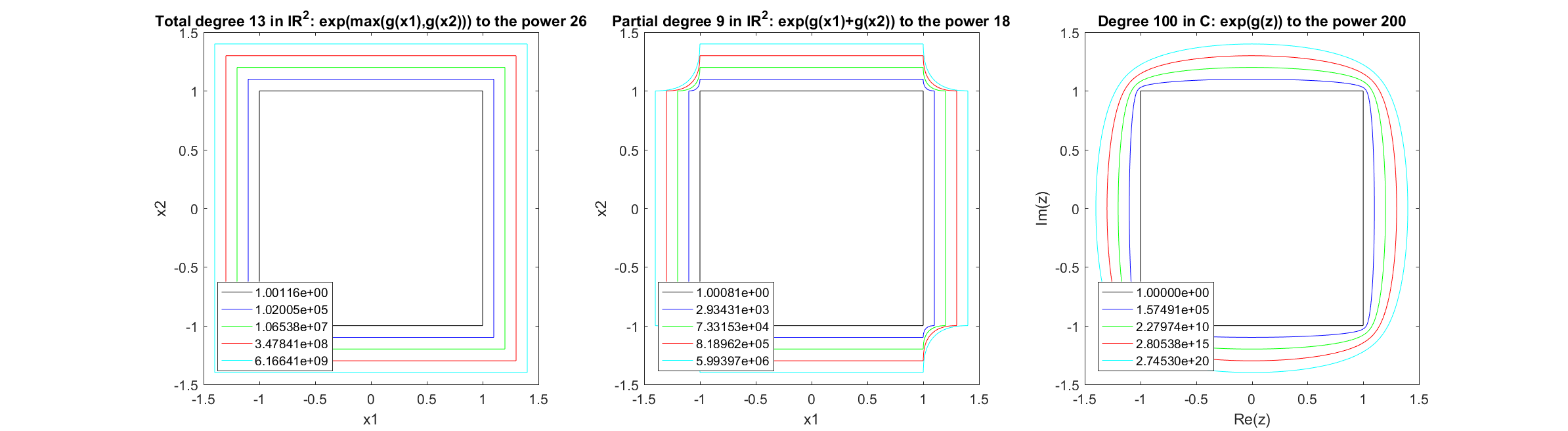}}
   \caption{Comparison of level lines of Green functions for the unit square and three different families of orthogonal polynomials: for the left-hand image we utilized the plurisubharmonic Green function~\eqref{plot_green1} of $[-1,1]^2 \subset \mathbb R^2,$ which corresponds to bivariate OP with graded lexicograpghical ordering (i.e., total degree); the middle image is generated by using the tensor Green function~\eqref{plot_green2} corresponding to bivariate OP and partial degree is used; and for the right-most image, the complex Green function corresponding to Bergman OP of one complex variable is used.}\label{fig7.2}
\end{figure}

   In Figure~\ref{fig7.2} we compare three approaches to detect outliers outside the unit square $[-1,1]^2$, and in particular address the question whether one should prefer an analysis in $\mathbb R^2$ or the complex plane $\mathbb C$. Though not fully justified for the case of two real variables, we expect to be able to detect successfully outliers at $z\not\in [-1,1]^2$ for a parameter $N$ provided that $K_N^\mu(z,z)$ is large. The left-hand plot corresponds to graded lexicographical ordering (total degree) discussed in the previous paragraph, where we draw some level lines of $\exp(2n\widetilde g(z))$ with the plurisubharmonic Green function \eqref{plot_green1} for $z\in \mathbb R^2\setminus [-1,1]^2$ instead of those of $K_N^\mu(z,z)$, $N+1=(n+1)(n+2)/2$. In the middle plot corresponding to the tensor case (partial degree) discussed in Theorem~\ref{thm_tensor} we draw some level lines of $\exp(2ng(z))$ with the tensor Green function \eqref{plot_green2} for $z\in \mathbb R^2\setminus [-1,1]^2$ instead of those of $K_N^\mu(z,z)$, $N+1=(n+1)^2$. Finally, on the right corresponding to the case of one complex variable $z$ discussed in \S\ref{sec_Bergman} one finds some level lines of $\exp(2ng_{\mathbb C\setminus[-1,1]^2}(z))$ for $z\in \mathbb C\setminus [-1,1]^2$ instead of those of $K_N^\mu(z,z)$, $N+1=n+1$, compare with Theorem~\ref{thm_Bergman}(c). Notice that there is no easy closed form expression for this complex Green function $g_{\mathbb C\setminus[-1,1]^2}$, we have used the Schwarz-Christoffel toolbox for Matlab of Toby Driscoll. In order to make comparison fair, we should use about the same number $N+1$ of orthogonal polynomials: we have chosen from the left to the right $n=13$, $n=9$ and $n=100$, corresponding to $N+1=105$, $N+1=100$ and $N+1=101$, respectively.

   Some observations are are especially noteworthy. We start by comparing the two bivariate kernels on the left-hand side and in the middle: from the behavior of the level lines at the corners of $[-1,1]^2$ it is apparent that one should prefer partial degree to total degree for detecting outliers with all components outside $[-1,1]$ (called the corner case in Theorem~\ref{thm_tensor}). Comparing the values of the corresponding level lines, the opposite conclusion seems to be true for other outliers. However, much more striking, the parameters of the level lines on the right-hand side are increasing much faster, and also here the level curves do fine around the corners.
   This confirms our claim that outlier detection of $(2d)$-dimensional data should be done in $\mathbb C^d$ and not in $\mathbb R^{2d}$, at least for $d=1$.

\subsection{Christoffel-Darboux kernel associated to Reinhardt domains in $\C^d$}\label{sec_Reinhardt} Natural domains of convergence for power series
in several complex variables are quite diverse, and in general they are not bi-holomorphically equivalent. These are known as
(logarithmically convex) Reinhardt domains and offer already a vast area of research for function theory and complex geometry.
The recent monograph \cite{JP} gives a comprehensive account of the theory of Reinhardt domains. For our immediate
purpose, this class of domains $D$ is important because the monomials are orthogonal (but not orthonormal) with respect to Lebesgue measure $\mu^D$ (and thus $k_n^{\mu^D}=\deg p_n^{\mu^D}=n$ in Lemma~\ref{lem_degree_OP}); moreover in
most cases the monomials are complete in the associated Bergman space.

By way of illustration we consider the two most important, and nonequivalent, Reinhardt domains: the ball and the polydisk. In the sequel of this subsection we will work in $\mathbb C^d$ for some fixed $d\geq 1$, and enumerate the multivariate monomials $z^\alpha$ in graded lexicographical ordering. Let $N=N(n)$ be the number of monomials of total degree $\leq n$.

\subsubsection{The complex ball} We first consider the Lebesgue (or volume) measure $\mu^{\B}$ on the unit ball
$$
    \B = \B^d = \{ z \in \C^d; \ \| z \| <1\}, $$
having as boundary the odd dimensional sphere $S^{2d-1}$. Here with help of polar coordinates one gets
$$
     \int z^\alpha \overline{z^\beta} d\mu^\B(z)  = \frac{\pi^d \alpha!}{(|\alpha|+d)!} \quad \mbox{for $\alpha=\beta$, and zero else.}
$$
Thus, up to normalization, the orthonormal polynomials are monomials, and we get for the Christoffel-Darboux kernel using the multinomial formula
$$
       K_N^{\mu^\B}(z,w) =
       \frac{1}{\pi^d} \sum_{k=0}^n \sum_{|\alpha|=k} \frac{(|\alpha|+d)!}{\alpha!} z^\alpha \overline{w}^\alpha
       =
       \frac{1}{\pi^d} \sum_{k=0}^n (k+1)_d (w^*z)^k,
$$
where we use the abbreviation $w^*z=\overline w_1 z_1+...+\overline w_d z_d$, together with the Pochhammer symbol $(a)_k=a(a+1)...(a+k-1)$.
Again, to simplify notation, above and henceforth in this section $N = n_{tot}(n).$

We already know from Siciak, see formula (\ref{complex-ball}), that
$$ \lim_{n \rightarrow \infty} K_N^{\mu^\B}(z,z)^{1/n} = \| z \|^2, \ \ \| z \| >1.$$
The explicit form of the polarized kernel allows a closer look at the asymptotic behavior for large degree.

Provided that $|w^*z|<1$ (and thus in particular for $w,z\in \B$), the limit for $N\to \infty$ and thus for $n\to \infty$ exists, and is given by the Bergman reproducing kernel
$$
        K^{\mu^\B}(z,w) =
               \frac{1}{\pi^d} \sum_{k=0}^\infty (k+1)_d (w^*z)^k
               = \frac{d!}{\pi^d} (1-w^*z)^{-d-1} .
$$
Note that the unbounded region defined in $\C^d \times \C^d$ by the inequality $|w^*z|<1$ may well project on one or the other coordinate in the exterior of the ball.

In the case $w^*z=1$ we obtain
$$
       K_N^{\mu^\B}(z,w) = \frac{d!}{\pi^d} \sum_{k=0}^n \Bigl( { {k+d } \atop d } \Bigr) = \frac{d!}{\pi^d} \Bigl( { {n+1+d } \atop {d+1} } \Bigr)
       = \frac{(n+1)_{d+1}}{\pi^d (d+1)} ,
$$
which is also an upper bound for $K_N^{\mu^\B}(z,w)$ in the case $|w^*z|=1$.
In particular we learn that $K_N^{\mu^\B}(z,z)$ grows at most like $n^{d+1}=\mathcal O(N^2)$ for $z$ in the support of $\mu^{\B}$. Finally, in the case $|w^*z|>1$,
$$
        K^{\mu^\B}_N(z,w) =
               \frac{(n+1)_d}{\pi^d} (w^*z)^n \sum_{j=0}^n \frac{(n-j+1)_d}{(n+1)_d} (w^*z)^{-j}
        \sim \frac{(n+1)_d}{\pi^d} \frac{(w^*z)^{n+1}}{{w^*z}-1}
$$
where in the last step we have used the fact that, for $0\leq j\leq n$,
$$
     0 \leq 1 - \frac{(n-j+1)_d}{(n+1)_d}
     = \sum_{\ell=0}^{j-1} \frac{(n-\ell+1)_d-(n-\ell)_d}{(n+1)_d}
     \leq \frac{jd}{n+d}.
$$
If we exclude the (unlikely for outlier analysis) case that $z$ is a complex multiple of $w$, we conclude that, for $z,w\not\in \supp(\mu^\B)$,
\begin{equation} \label{eq.Ball}
      K^{\mu^\B}_N(z,z)
              \sim \frac{(n+1)_d}{\pi^d} \frac{\| z\|^{2n+2}}{\| z \|^2-1},
              \quad \lim_{N\to \infty} C_N(z,w) = 0 .
\end{equation}

\subsubsection{The polydisk} In complete analogy, consider the Lebesgue measure $\mu^{\PP}$ on the polydisk $\PP = \D^d$. Here again multivariate monomials are orthogonal, with the normalization constant
$$
      \| z^\alpha \|^2_{2,\mu^{\PP}} = \frac{\pi^d}{\prod_{j=1}^d (\alpha_j+1)},
$$
and hence
$$
      K_N^{\mu^{\PP}}(z,w) =
      \frac{1}{\pi^d} \sum_{k=0}^n \sum_{|\alpha|=k} z^\alpha \overline w^\alpha \prod_{j=1}^d (\alpha_j +1 ).
$$
It is not difficult to check that, provided that $|w_jz_j|<1$ for all $j$,
$$
          \lim_{N\to \infty}
          K_N^{\mu^{\PP}}(z,w) =
          \frac{1}{\pi^d} \prod_{j=1}^d \frac{1}{(1-z_j \overline w_j)^2},
$$
which is the well known Bergman space reproducing kernel for $\PP = \mathbb D^d$. Also, the reader may check that, again, $K_N^{\mu^{\PP}}(z,z)$ grows at most polynomially in $N$ for $z\in \supp(\mu^{\mathbb D^d})$. As for the exterior behavior, again the extremal plurisubharmonic function approach gives
$$
\lim_{n \to \infty} K_N^{\mu^{\PP}}(z,z)^{1/n} = \max_{1 \leq j \leq d} |z_j|^2, \ \ z \notin \PP,
$$
cf. formula (\ref{product}).

Since the further analysis is a bit involved, we will restrict ourselves to the special case $d=2$ and $z,w\not\in \PP$; that is, $\max (|z_1|,|z_2|) >1$ and
$\max (|w_1|,|w_2|) >1$. If $|\overline w_1z_1|<|\overline w_2 z_2|$ then
\begin{eqnarray*}
      \pi^2 K_N^{\mu^{\PP}}(z,w) &=&
      \sum_{k=0}^n (\overline w_2z_2)^k \sum_{j=0}^k (j+1)(k+1-j) \bigl( \frac{\overline w_1z_1}{\overline w_2z_2}\Bigr)^j
      \\&\sim&
      \sum_{k=0}^n (k+1) (\overline w_2z_2)^k
      \frac{(\overline w_2z_2)^2}{(\overline w_2z_2-\overline w_1z_1)^2}
      \\&\sim&
      \frac{n+1}{\overline w_2z_2-1}
      \frac{(\overline w_2z_2)^{n+3}}{(\overline w_2z_2-\overline w_1z_1)^2}.
\end{eqnarray*}
The case $\overline w_1z_1=\overline w_2z_2$ can be similarly treated. We conclude that $K_N^{\mu^{\PP}}(z,z)$ grows outside the support at least like
$\max \{ |z_1|^{2n},|z_2|^{2n}\}$ times some polynomial in $n$. An asymptotic analysis for $C_N^{\mu^{\PP}}(z,w)$ is possible but quite involved and we omit the details.

\section{Index of notation}\label{sec_index} {\small
We list here the meanings of many of the symbols that are
 frequently used throughout the paper.

\begin{itemize}

\item $K^\mu_n(z,w)$ the Christoffel-Darboux kernel consisting of $(n+1)$ summands;

\item ${\rm tdeg} p$ the total degree of a multivariate polynomial $p$;

\item $\deg p$ the degree of a multivariate orthogonal polynomial, indicating the position $n = \deg p$ of $p$ in a prescribed linear ordering of independent monomials;

\item $S = \supp(\mu)$  the support of the measure $\mu$;

\item $\mathcal S (\mu)$ the Zariski closure of the support of a positive measure $\mu$;

\item $\mathcal N (\mu)$ the ideal of polynomials vanishing on $\mathcal S (\mu)$;

\item $v_n(z)$ the tautologic vector of monomials of degree less than or equal to $n$;

\item $v^\mu_n(z)$ the tautologic vector of $\mu$-orthogonal polynomials;

\item $\widetilde{M}_n(\mu)$ the matrix of moments of degree less than or equal to $n$, associated to a positive measure $\mu$;

\item $M_n(\mu)$ the reduced moment matrix, a maximal invertible submatrix of $\widetilde{M}_n(\mu)$;

\item $\Delta(z,\mu)$ the Mahalanobis distance between a point $z$ and a measure $\mu$;

\item $t_j K_n^{\widetilde \nu}(z^{(j)},z^{(j)})$ the leverage score of a masspoint $z^{(j)}$ with mass $t_j$ of a discrete measure $\widetilde \nu$;

\item $\mathcal L_n (\mu)$ the linear span of orthogonal polynomials of degree less than or equal to $n$;

\item $C^\mu_n(z,w)$ the cosine function associated to the Christoffel-Darboux kernel;

\item $\mathcal{C}(\mu)$ the covariance matrix of a random variable with law $\mu$;

\item $g_\Omega$ the plurisubharmonic Green function of the unbounded open set $\Omega$, with pole at infinity;

\item $\Phi$ the normalized conformal mapping of the unbounded component of the complement of a Jordan curve in $\C$ onto the exterior of the unit disk;

\item $T_n$ Chebyshev polynomial of the first kind.

\end{itemize}
}


\end{document}